\documentclass[12pt,reqno]{amsart}

\usepackage{subfiles}
\usepackage{amsmath,amssymb,amsfonts,amscd,latexsym,amsthm,mathrsfs,verbatim, hyperref,tikz-cd}
\usepackage{graphicx}
\usepackage[doi=false, isbn=false, style=trad-alpha]{biblatex}
\usepackage{fullpage}

\renewcommand{\P}{\mathbb{P}}

\renewcommand{\div}{\operatorname{div}}

\DeclareMathOperator{\Hom}{Hom}

\DeclareMathOperator{\spec}{Spec}

\DeclareMathOperator{\gal}{Gal}

\DeclareMathOperator{\Bl}{Bl}
\DeclareMathOperator{\rank}{rank}
\DeclareMathOperator{\mult}{mult}
\DeclareMathOperator{\mon}{mon}

\DeclareMathOperator{\coker}{coker}
\DeclareMathOperator{\ev}{ev}
\DeclareMathOperator{\inv}{inv}
\DeclareMathOperator{\pr}{pr}

\DeclareMathOperator{\Supp}{Supp}

\renewcommand\d{\,\mathrm d}

\DeclareMathOperator{\Pic}{Pic}

\DeclareMathOperator{\Eff}{Eff}
\DeclareMathOperator{\Br}{Br}

\newtheorem{theorem}{Theorem}
\newtheorem{proposition}[theorem]{Proposition}
\newtheorem{lemma}[theorem]{Lemma}

\newtheorem{conjecture}[theorem]{Conjecture}

\theoremstyle{definition}
\newtheorem{definition}[theorem]{Definition}
\newtheorem{example}[theorem]{Example}
\newtheorem{remark}[theorem]{Remark}

\newtheorem{assumption}[theorem]{Assumption}
\newtheorem{notation}[theorem]{Notation}

\numberwithin{theorem}{section}
\numberwithin{equation}{section}

\newcommand{\lineN}{\overline{\mathbb{N}}}
\newcommand{\cA}{\mathcal{A}}

\newcommand{\cC}{\mathcal{C}}

\newcommand{\cO}{\mathcal{O}}
\newcommand{\cD}{\mathcal{D}}

\newcommand{\cU}{\mathcal{U}}

\newcommand{\cM}{\mathcal{M}}
\newcommand{\fM}{\mathfrak{M}}

\newcommand{\cX}{\mathcal{X}}
\newcommand{\cY}{\mathcal{Y}}

\newcommand{\cL}{\mathcal{L}}
\newcommand{\cK}{\mathcal{K}}

\newcommand{\fp}{\mathfrak{p}}

\newcommand{\bfm}{\mathbf{m}}
\newcommand{\bfw}{\mathbf{w}}

\newcommand{\bfx}{\mathbf{x}}
\newcommand{\bfe}{\mathbf{e}}

\newcommand{\fin}{\textnormal{fin}}
\newcommand{\red}{\textnormal{red}}

\DeclareMathOperator{\Div}{Div}
\DeclareMathOperator{\Volume}{Volume}

\newcommand{\lParen}{(\!(}
\newcommand{\rParen}{)\!)}

\DeclareFontFamily{U}{wncy}{}
\DeclareFontShape{U}{wncy}{m}{n}{<->wncyr10}{}
\DeclareSymbolFont{mcy}{U}{wncy}{m}{n}
\DeclareMathSymbol{\Sh}{\mathord}{mcy}{"58} 
\DeclareMathSymbol{\B}{\mathord}{mcy}{"42}

\begin{document}
\begin{abstract}
	We initiate a general quantitative study of sets of $\cM$-points, which are special subsets of rational points, generalizing Campana points, Darmon points, and squarefree solutions of Diophantine equations. We propose an asymptotic formula for the number of $\cM$-points of bounded height on rationally connected varieties, extending Manin's conjecture as well as its generalization to Campana points by Pieropan, Smeets, Tanimoto and Várilly-Alvarado. Finally, we show that the conjecture explains several previously established results in arithmetic statistics.
\end{abstract}
	\title{Manin's conjecture for $\cM$-points}
	\author{Boaz Moerman}
	
	\maketitle

	\setcounter{tocdepth}{1}
	\tableofcontents

	\section{Introduction}
Over the last few decades there has been a focused effort to count the number of rational points on algebraic varieties and special subsets thereof. Manin's conjecture is a major source of inspiration for many of these developments, which gives an asymptotic formula for the number of rational points of bounded height on rationally connected varieties. This conjecture has been refined over several decades \cite{FMT89, BaMa90,Pey95, BaTs98,LeSeKuSh22}, and it has been established for several families of varieties (see e.g. \cite{BaTs96, GoMaOh08, Woo24}).

In \cite{PSTVA21}, Pieropan, Smeets, Tanimoto and Várilly-Alvarado proposed an asymptotic formula for the number of Campana points of bounded height, extending Manin's conjecture. Campana points are rational points which intersect a given collection of boundary divisors $D_1,\dots,D_n$ with high multiplicity, and they arise naturally as powerful solutions of Diophantine equations. Their conjecture has been established in various special cases, as we will discuss below.

Besides Campana points, various other related subsets have been studied recently, including weak Campana points, geometric Campana points and (geometric) Darmon points. Several asymptotic counting results have been established for these points \cite{Str21,ShSt24, Ito25, Ara25}, in the spirit of Manin's conjecture. Despite these results, no extension of Manin's conjecture to these types of points has been proposed thus far.

These various special sets of rational points mentioned are all instances of sets of $\cM$-points, which were first introduced in \cite{Moe24}. Many other arithmetically interesting sets of rational points can be viewed as $\cM$-points, such as the set of squarefree solutions to a given Diophantine equation. Similarly, the set $$\{(a_1:\dots:a_n)\in \mathbb{P}^{n-1}(\mathbb{Z})\mid  f(a_1,\dots,a_n) \text{ squarefree}\}$$ is also a set of $\cM$-points, for any given homogeneous integer polynomial $f$. Such sets are related to the study of squarefree values of polynomials, which is an enduring and active research area, see for example \cite{Hoo67, Fil92, Poo03, SaWa23}.

In this article we introduce Conjecture \ref{conj: generalized Manin conjecture}, which gives an asymptotic formula for the number of $\cM$-points of bounded height, further extending Manin's conjecture. We show that the conjecture agrees with the prediction given for Campana points in \cite{PSTVA21}, with the corrected leading constant formulated in \cite[Conjecture 8.3]{CLTT24}.

Our conjecture also gives an asymptotic for Campana points of bounded height for more Campana pairs than the conjecture in \cite{PSTVA21}.
 One major limitation in the applicability of the latter conjecture is that it requires the boundary to be a strict normal crossings divisor. In contrast, our conjecture only assumes that the boundary is a normal crossings divisor after passing to the algebraic closure of the base field. In particular, our conjecture explains the asymptotic formula \cite{Str21} found by Streeter for the number of powerful values of norm forms. His results did not match the asymptotic formula given by (a naive extension) of the conjecture \cite{PSTVA21}, but his results agree with our new conjecture, as we will show.
 
 More generally, the weakened hypothesis makes Conjecture \ref{conj: generalized Manin conjecture} applicable in much more situations than the previous conjecture. For instance, our conjecture gives a prediction for the number of Campana points of bounded height on a toric variety (with torus-invariant boundary), while the conjecture in \cite{PSTVA21} is only applicable for special toric varieties such as split toric varieties.

Furthermore, our conjecture also predicts an asymptotic for the number of weak Campana points of bounded height. In \cite{PSTVA21}, a preliminary investigation was done into such asymptotics, but the authors concluded that a conjecture for them would require a new set of ideas.

Finally, we verify that the conjecture agrees with the various results \cite{Str21,ShSt24, Ito25, Ara25} known on $\cM$-points. We will treat Streeter's result \cite[Theorem 1.1]{Str21} on powerful values of norm forms in great detail, to show it agrees with our prediction. In particular, this gives a geometric interpretation for the exponent on the logarithm appearing in the asymptotic, which was not known previously \cite[Remark 1.5]{Str21}. In another work \cite{Moe25toric}, we prove this conjecture for split toric varieties over the rational numbers, giving the first general results on counting $\cM$-points.

\subsection{Manin's conjecture and Campana points}
Before stating our conjecture, we first give a short overview of Manin's conjecture and its generalization to Campana points.

If $X$ is a variety over a number field and $\cL$ is an adelically metrized line bundle on $X$, then we write $$H_{\cL}\colon X(K)\rightarrow \mathbb{R}_{>0}$$ for the height on $X$ as defined in \cite[\S 1.3]{Pey95}. We denote the corresponding counting function by $$N(A,\cL,B)=\#\{P\in A\mid H_{\cL}(P)\leq B\},$$ where $A\subset X(K)$ and $B$ is an integer. We recall the most recent version of Manin's conjecture conjecture, which is given for example in \cite[Conjecture 1.2]{LeSeKuSh22}.
\begin{conjecture}[Manin's conjecture] \label{conjecture: Manins conjecture}
	Let $X$ be a proper smooth rationally connected variety over a number field $K$ and assume that $X(K)$ is not a thin set. Then there exists a thin set $Z\subset X(K)$ such that 
	$$N(X(K)\setminus Z,\cL,B)\sim cB^{a(X,L)}(\log B)^{b(K,X,L)-1} \quad \text{ as } B\rightarrow \infty,$$
	where $c>0$ is an explicit Tamagawa constant, $$a(X,L)=\inf\{t\in \mathbb{R}\mid tL+K_X\in \overline{\Eff}^1(X)\}$$ is the Fujita invariant of $X$ with respect to $L$ and $b(K,X,L)$ is the codimension of the minimal face of the pseudo-effective cone $\overline{\Eff}^1(X)$ containing $a(X,L)L+K_X$.
\end{conjecture}
Manin's conjecture was first formulated and studied for Fano varieties with an anticanonical height in 1989 and 1990 by Manin, Batyrev, Tschinkel and Franke \cite{FMT89, BaMa90}. Peyre \cite{Pey95} further contributed to the conjecture by giving a conjectural value for the constant $c$. Later, Batyrev and Tschinkel extended the conjecture to heights induced by other divisors. More recently Lehmann, Sengupta and Tanimoto \cite[Conjecture 5.2]{LeSeKuSh22} have formulated a prediction for the thin set $Z$ that has to be excluded.

For Campana points, a similar conjecture has been formulated. Recall that (klt) Campana points on a smooth projective variety $X$ are defined using a Campana pair $(X,D_{\bfm})$, where $$D_{\bfm}=\sum_{i=1}^n \left(1-\tfrac{1}{m_i}\right)D_i$$ is a boundary divisor given by weights $m_1,\dots,m_n\in \mathbb{N}$ and prime divisors $D_1,\dots,D_n$ on $X$. Let $S\subset \Omega_K$ be a finite set of places of $K$ including the infinite places and let $\cX$ be an $\cO_S$-integral model of $X$. Let  $(\cX,\cD_{\bfm})$ be the corresponding $\cO_S$-integral model of $(X,D_{\bfm})$ and write $(\cX,\cD_{\bfm})(\cO_S)$ for its set of Campana points over $\cO_S$. In \cite{PSTVA21}, Pieropan, Smeets, Tanimoto and Várilly-Alvarado formulated the following conjecture on the number of Campana points of bounded height.
\begin{conjecture} \label{conjecture: Manin Campana}
	Let $X$ be a smooth Fano variety and assume that the log-anticanonical divisor class $-K_X-D_{\bfm}$ is ample and that the support $\Supp D_{\bfm}$ is a strict normal crossings divisor. Furthermore let $\cX$ be a regular $\cO_S$-model $\cX$ for which $(\cX,\cD_{\bfm})(\cO_S)$ is not a thin set. Then there exists a thin set $Z\subset (\cX,\cD_{\bfm})(\cO_S)$ such that 
	$$N((\cX,\cD_{\bfm})(\cO_S)\setminus Z,\cL,B)\sim cB^{a((X,D_{\bfm}),L)}(\log B)^{b(K,(X,D_{\bfm}),L)-1} \quad \text{ as } B\rightarrow \infty,$$
	where $c=c(K,S,(\cX,\cD_{\bfm}),\cL)>0$ is a constant,
	$$a((X,D_{\bfm}),L)=\inf\{t\in \mathbb{R}\mid tL+K_X+D_{\bfm}\}$$ is the Fujita invariant of $(X,D_{\bfm})$ with respect to $L$, and $b(K,(X,D_{\bfm}),L)$ is the codimension of the minimal face of $\overline{\Eff}^1(X,M)$ containing $a((X,D_{\bfm}),L)L+K_X+D_{\bfm}$.
	
	Furthermore, if this divisor class is rigid, then $c$ is an explicit Tamagawa constant, given in \cite[Conjecture 8.3]{CLTT24}.
\end{conjecture}
The original conjecture, as given in \cite{PSTVA21}, had a different prediction for the leading constant. However, Shute \cite{Shu22} and Streeter \cite{Str21} independently gave counterexamples to this prediction. In \cite{CLTT24}, Chow, Loughran, Takloo-Bighash and Tanimoto modified the conjecture by changing the Tamagawa constant to the one cited above, so that Shute's and Streeter's results are in agreement with their conjecture. Their constant is defined very similarly to the original constant, but it is defined using a different Brauer group which we recall in Section \ref{section: fundamental group and Brauer group}.

Conjecture \ref{conjecture: Manin Campana} has been proven for various Campana pairs. These include results on diagonal hypersurfaces \cite{Val12, BrYa21,Shu21,Shu22,BBKOPW23}, vector group compactifications \cite{PSTVA21}, norm forms \cite{Str21}, biequivariant compactifications of the Heisenberg group \cite{Xia21} and wonderful compactifications \cite{CLTT24}. Furthermore, for the log-anticanonical height, this conjecture has been proven for complete toric varieties \cite{PiSc23, ShSt24} and certain complete intersections therein \cite{PiSc24}.

Furthermore, in \cite{Fai23, Fai25} Faisant proves a motivic analogue of the conjecture for Campana points for vector group compactifications and toric varieties.


The assumption in Conjecture \ref{conjecture: Manin Campana} that $\Supp D_{\bfm}$ is a strict normal crossings divisor is rather restrictive, as many interesting Campana pairs do not satisfy this hypothesis. For instance, this assumption need not be satisfied if $X$ is a smooth projective toric variety and $\Supp D_{\bfm}$ is the complement of the open torus. Nevertheless, this assumption is important as Streeter shows in \cite{Str21} that the logarithmic factor appearing in the asymptotic is larger than the conjecture would suggest. In order to avoid this issue, Shute and Streeter introduced the notion of geometric Campana points \cite{Str21,ShSt24}, which are a variant on Campana points defined using the geometric components of $D_{\bfm}$ (rather than those defined over $K$). They showed that Conjecture \ref{conjecture: Manin Campana} holds for the \textit{geometric} Campana points on toric Campana pairs with respect to the log-anticanonical height.

In a similar fashion, analogues of Conjecture \ref{conjecture: Manin Campana} have been proven for (geometric) Darmon points \cite{ShSt24, Ito25, Ara25}.
	
\subsection{\texorpdfstring{$\cM$}{M}-points}
One thing in common between Campana points, Darmon points and their variants, is that they are defined as rational points satisfying certain conditions on their intersection multiplicities with the boundary components.
In this article, we will extend Conjecture \ref{conjecture: Manin Campana} to the setting of $\cM$-points. $\cM$-points are a broad generalization of Campana points, Darmon points, as well as their variants mentioned above. They were first introduced by the author in \cite{Moe24}. Like Campana points, $\cM$-points are rational points satisfying conditions on their intersection multiplicities with the boundary components.

Similar to how Campana points are points on an integral model of a Campana pair $(X,D_{\bfm})$, $\cM$-points can be viewed a points on an integral model of a pair $(X,M)$. Here the pair $(X,M)$ consists of a proper variety $X$ together with a parameter set $M$, the latter of which is often given by a tuple of divisors $(D_1,\dots,D_n)$ of divisors on $X_{\overline{K}}$ together with a multiplicity set $\fM\subset (\mathbb{N}\cup \{\infty\})^n$ of allowed intersection multiplicities with these divisors.
\begin{example}
	For instance, if $D_1,\dots,D_n$ are the coordinate hyperplanes of $X=\mathbb{P}^{n-1}$, then the set of $\cM$-points over $\mathbb{Z}$ is
	$$\{(a_1:\dots:a_n)\in \mathbb{P}^n(\mathbb{Q})\mid (v_p(a_1),\dots,v_p(a_n))\in \fM\text{ for all prime numbers } p\}.$$
\end{example}
More generally, we also allow the set $M$ to impose conditions with respect to the different intersection strata, as we discuss in Section \ref{section: pairs and M-points}.  The notion of pair used in this article is more general than the notion considered in \cite{Moe24, Thesis}, as we explain in Section \ref{section: pairs and M-points}, where we define pairs and $\cM$-points.

\subsection{\texorpdfstring{$\cM$}{M}-points of bounded height}
	Now we will formulate a version of Manin's conjecture for $\cM$-points.
	Let $(X,M)$ be a smooth, proper and rationally connected pair over a number field $K$, as defined in Section \ref{section: Manin's conjecture for M-points}. For such a pair, we introduce its Picard group $\Pic(X,M)$ along with a natural group homomorphism $\pr^*_M \colon \Pic(X)\rightarrow \Pic(X,M)$, its pseudo-effective cone $\overline{\Eff}^1(X,M)$ as well as its canonical divisor class $K_{(X,M)}\in \Pic(X,M)$. Using these notions, we formulate our conjecture in exact analogy to Manin's conjecture. Fix a finite set $S$ of places of $K$, including all infinite places and the places that ramify in the splitting field of $(X,M)$. For an integral model $(\cX,\cM)$ of $(X,M)$ over $\cO_S$, we denote the set of $\cM$-points on $(\cX,\cM)$ over $\cO_S$ by $(\cX,\cM)(\cO_S)\subset X(K)$. As before, we let $L$ be a big and nef divisor class with adelic metrization $\cL$.
	\begin{conjecture} \label{conj: generalized Manin conjecture}
		Let $(X,M)$ be a smooth proper pair over a number field $K$ such that $(X,M)$ is rationally connected, and let $(\cX,\cM)$ be an integral model of $(X,M)$ over $\cO_S$. Assume furthermore that $(\cX,\cM)(\cO_S)\subset X(K)$ is Zariski-dense in $X$. Then there exists a thin set $Z\subset X(K)$ such that
		$$N((\cX,\cM)(\cO_S)\setminus Z,\cL,B)\sim cB^{a((X,M),L)}(\log B)^{b(K,(X,M),L)-1} \quad \text{ as } B\rightarrow \infty,$$
		where $c>0$ is a constant, $$a((X,M),L)=\inf\{t\in \mathbb{R}\mid t\pr^*_M L+K_{(X,M)}\}$$ is the Fujita invariant of $(X,M)$ with respect to $L$, and $b(K,(X,M),L)$ is the codimension of the minimal face of $\overline{\Eff}^1(X,M)$ containing $a((X,M),L)\pr^*_M L+K_{(X,M)}$. 
		
		Furthermore, assume that that $L$ is adjoint rigid with respect to $(X,M)$ as in Definition \ref{def: adjoint rigid}, and that the fundamental group $\pi_1(X^{\circ}_{\overline{K}},M^{\circ})$ is abelian (see Section \ref{section: leading constant general}). Then the constant $c=c(K,S,(\cX,\cM),\cL)$ is a Tamagawa constant explicitly given in Section \ref{section: leading constant general}.
	\end{conjecture}
This conjecture is an extension of the conjecture \cite[Conjecture 1.2.2]{Thesis} presented in the author's PhD thesis, where only split pairs were considered and no prediction for the leading constant was given.
\begin{remark}
The assumption that $(X,M)$ is smooth is a much more relaxed assumption than the assumption in Conjecture \ref{conjecture: Manin Campana} that $D_{\bfm}$ has strict normal crossings support. For instance, the conjecture can be applied to any smooth toric variety, where the boundary is taken to be the complement of the torus.
\end{remark}
	\begin{remark}
		The assumption on the fundamental group is satisfied whenever the complement $U$ of the boundary divisors has an abelian fundamental group (which is satisfied in all cases for which the conjecture is known). In personal communication, Tim Santens has suggested a Campana pair for which we expect that the prediction in Conjecture \ref{conj: generalized Manin conjecture} for the constant fails due to the existence of a nonabelian cover. This assumption was also suggested to be necessary in \cite{CLTT24} for the same reason, see \cite[Remark 8.4]{CLTT24}.
	\end{remark}
	Conjecture \ref{conj: generalized Manin conjecture} can be directly seen to be a generalization of Conjecture \ref{conjecture: Manins conjecture}. It is not as apparent that it also generalizes Conjecture \ref{conjecture: Manin Campana}, but this follows from the next result, which gives a description of the Fujita invariant and $b$-invariant for (geometric) Campana points, (geometric) Darmon points, and weak Campana points. 

\begin{theorem} \label{theorem: invariants quasi-Campana}
	Let $(X,D_{\bfm})$ be a Campana pair over a number field $K$ such that $X$ is smooth and its effective cone is a rational polyhedral cone. Furthermore assume that there is a field extension $E/K$ such that $$D_{\bfm,E}=\sum_{i=1}^n \left(1-\frac{1}{m_i}\right)D_i$$ is a strict normal crossings divisor, for prime divisors $D_1,\dots,D_n$. Let $(X,M)$ be the pair corresponding to one of the following special points on the Campana pair $(X,D_{\bfm})$: (geometric) Campana points, (geometric) Darmon points, weak Campana points. 
	
Then the Fujita invariant appearing in Conjecture \ref{conj: generalized Manin conjecture} is given by $$a((X,M),L)=\inf\{t\in \mathbb{R}\mid tL+K_X+D_{\bfm}\in \overline{\Eff}^1(X)\}.$$
	The description of the $b$-invariant in the conjecture depends on the type of $\cM$-points considered. Let $b(K,(X,D_{\bfm}),L)$ be the codimension of the minimal face of the pseudo-effective cone $\overline{\Eff}^1(X)$ containing $a((X,M),L)L+K_X+D_{\bfm}$.
	\begin{itemize}
		\item If $(X,M)$ is the pair corresponding to geometric Campana(/Darmon) points, then $$b(K,(X,M),L)=b(K,(X,D_{\bfm}),L).$$
		\item If $(X,M)$ is the pair corresponding to weak Campana points, then 
		$$b(K,(X,M),L)=b(K,(X,D_{\bfm}),L)+\#B/G.$$
		Here $G=\gal(E/K)$ and $B$ is the set of all tuples $(\bfw,c)$ such that
		\begin{enumerate}
			\item $\bfw\in \mathbb{N}^n$ with $\sum_{i=1}^n \frac{w_i}{m_i}=1$, with $\bfw\neq w_1\bfe_1,\dots,w_n\bfe_n$,
			\item $w_i=0$ for all $i\in \{1,\dots,n\}$ such that $D_i$ appears in the support of $$a((X,M),L)L+K_X+D_{\bfm},$$
			\item $c$ is an irreducible component of
			$\bigcap_{w_i>0}D_i.$
		\end{enumerate}
		The action of $G$ on $\mathbb{N}^n$ here is induced by the action of $G$ on $\{D_1,\dots,D_n\}$.
		
		\item If $(X,M)$ is the pair corresponding to Campana points or Darmon points, then $$b(K,(X,M),L)=b(K,(X,D_{\bfm}),L)+\#B'/G,$$
		where $B'\subset B$ is the subset of all $(\bfw,c)$ such that $\{D_i\mid i\in\{1,\dots,n\},\, w_i>0\}$ is contained in a Galois orbit.
	\end{itemize}
	Furthermore, the divisor $a((X,M),L)L+K_X+D_{\bfm}$ is rigid if and only if $L$ is adjoint rigid with respect to $(X,M)$.
\end{theorem}
This theorem follows from combining Proposition \ref{prop: Fujita invariant quasi-Campana}, Proposition \ref{prop: b-invariant quasi-Campana} and Proposition \ref{prop: rigid quasi-Campana}.\begin{remark}
 The condition on the effective cone of $X$ is quite weak, as it is satisfied if $X$ is a Mori dream space, such as when $X$ is a Fano variety \cite[Corollary 1.3.2]{BCHM10} or a toric variety.
\end{remark}

Theorem \ref{theorem: invariants quasi-Campana} has a variety of consequences. Firstly, it shows that Conjecture \ref{conjecture: Manin Campana} is a special case of Conjecture \ref{conj: generalized Manin conjecture}, provided that the relevant pair $(X,M)$ is rationally connected. This rational connectedness follows from a conjecture of Campana \cite[Conjecture 9.10]{Cam11}, which we will extend in Section \ref{section: quasi-Campana pairs} to other pairs. In Section \ref{section: Compatibility}, we will explain why the Tamagawa constant in our conjecture agrees with the Tamagawa constant in \cite{CLTT24}.

Furthermore, the Theorem \ref{theorem: invariants quasi-Campana} shows that Conjecture \ref{conj: generalized Manin conjecture} predicts a different rate of growth for the counting functions for geometric Campana points and Campana points of bounded height on a Campana pair $(\cX, \cD_{\bfm})$, provided that the set $B'$ appearing in the theorem is nonempty.

Furthermore, the theorem shows that the $b$-invariants for geometric Campana points and Campana points deviate for Campana pairs $(X,D_{\bfm})$, provided that the support of $D_{\bfm}$ is only a geometric strict normal crossings divisor. While for geometric Campana points $b(K(X,M),L)$ is the same as the invariant $b(K,(X,D_{\bfm}),L)$ appearing in Conjecture \ref{conjecture: Manin Campana}, the $b$-in for Campana points $b(K(X,M),L)$ is strictly bigger.


	
	\begin{remark}
		As described in \cite[\S 8]{Moe24}, Darmon points correspond to integral points on the corresponding root stack. Conjecture \ref{conj: generalized Manin conjecture} therefore gives a prediction for the asymptotic number of integral points of bounded height on a root stack $(\cX,\sqrt[\bfm]{\cD})$, where the height is induced by any metrized big and nef line bundle on $X$. This is related to the conjecture by Ellenberg, Satriano and Zureick-Brown on the number of rational points of bounded height on stacks \cite[Conjecture 4.14]{ESZB23} and its generalization by Darda and Yasuda \cite[Conjecture 9.16]{DaYa24-general}. Their conjectures use different heights however: the heights we consider are what Darda and Yasuda call unstable heights, while their conjecture uses stable heights instead.
	\end{remark}
Let us also remark some differences with the previous conjectures.
	\begin{remark}
		In contrast to Conjecture \ref{conjecture: Manins conjecture} and Conjecture \ref{conjecture: Manin Campana}, we do not assume that the set of $\cM$-points in Conjecture \ref{conj: generalized Manin conjecture} is not thin, but we only require that it is Zariski-dense. The reason for this is that by \cite[Theorem 1.1]{Moe25toric}, Conjecture \ref{conj: generalized Manin conjecture} holds for every smooth proper toric pair over $\mathbb{Q}$, even though there are many such pairs $(X,M)$ for which the set of $\cM$-points is thin by \cite[Theorem 6.5]{Moe24}. Furthermore, there are no known examples of rationally connected varieties with a thin, but nonempty, set of rational points. Any such example would contradict an open conjecture by Colliot-Thélène \cite[Conjecture 3.5.8]{Ser16}, \cite{CoT88}. 
	\end{remark}
	\begin{remark}
		In Conjecture \ref{conjecture: Manin Campana}, the assumption is made that the integral model $\cX$ is regular, which we do not assume in Conjecture \ref{conj: generalized Manin conjecture}. We suspect that this assumption is superfluous, supported by the results by Santens \cite[Theorem 1.1]{San23-2} providing an analogue of Manin's conjecture for integral points on toric varieties with respect to \textit{any} flat integral model.
	\end{remark}
	\subsection{Evidence for the conjecture}
	In \cite{Moe25toric}, we prove Conjecture \ref{conj: generalized Manin conjecture} for split toric varieties, providing the first results on $\cM$-points of bounded height in general.
	\begin{theorem}[{\cite[Theorem 1.1]{Moe25toric}}]
		Let $(X,M)$ be a smooth proper pair over $\mathbb{Q}$, where $X$ is a smooth proper split toric variety and the divisors $D_1,\dots,D_n$ are torus-invariant. Let $\cX$ be the $\mathbb{Z}$-integral model of $X$ induced by the fan of $X$ and let $(\cX,\cM)$ be the corresponding integral model of $(X,M)$. Then Conjecture \ref{conj: generalized Manin conjecture} is true for $(X,M)$ for any big and nef divisor class $L$ on $X$ with toric metrization $\cL$ (as defined in \cite[Theorem 2.1.6]{BaTs96}).
	\end{theorem}
In Section \ref{section: Compatibility}, we give further evidence for the conjecture by showing that the conjecture is compatible with the the various results on $\cM$-points of bounded height \cite{Str21,ShSt24, Ito25, Ara25}. In Theorem \ref{Theorem: conjecture compatible with norm forms} we give an asymptotic formula for the number of elements with $m$-full norm in any Galois extension of $\mathbb{Q}$, under assumption of Conjecture \ref{conj: generalized Manin conjecture}. Furthermore, we show that this asymptotic agrees with the result by Streeter \cite[Theorem 1.1]{Str21}, whenever the hypotheses of his result are satisfied. In order to compare our leading constant with the one found by Streeter, we prove Lemma \ref{lemma: Brauer group toric pair}, giving a description of the algebraic Brauer group of pairs $(X,M)$ over toric varieties in terms of automorphic characters.

\subsection{Structure of the paper}
In Section \ref{section: pairs and M-points}, we introduce pairs and $\cM$-points, as well as various special cases of these such as Campana points. In Section \ref{section: Picard group pair}, we introduce the Picard group of a pair, as well as its effective cone and its canonical divisor. In the same section, we also prove several results to aid in the computation of the Fujita invariant and the $b$-invariant. In Section \ref{section: Manin's conjecture for M-points} we define when a pair is rationally connected, and motivate its role in Conjecture \ref{conj: generalized Manin conjecture}. Section \ref{section: quasi-Campana pairs} is devoted to the study of the log-canonical divisor and quasi-Campana pairs, which are pairs loosely resembling the pairs for Campana points. For such pairs, we compute the Fujita invariant and the $b$-invariant, yielding Theorem \ref{theorem: invariants quasi-Campana} as a special case. In Section \ref{section: fundamental group and Brauer group}, we introduce the fundamental group and the Brauer group of a pair, and we study the latter for pairs over toric varieties. In Section \ref{section: leading constant general}, we define the leading constant in Conjecture \ref{conj: generalized Manin conjecture} using the Brauer group from the previous section. Finally, in Section \ref{section: Compatibility}, we prove that the previous results known on $\cM$-points of bounded height are compatible with Conjecture \ref{conj: generalized Manin conjecture}.

\subsection*{Acknowledgments}
This article is based on my PhD thesis \cite{Thesis}, which I wrote at Utrecht University under the supervision of Marta Pieropan. I would like to thank her for her unwavering support for my work and her various suggestions. I would also like to thank Alec Shute, Sam Streeter, Tim Santens and Damaris Schindler for the useful discussions we had during the workshop on Campana points on toric varieties in Bristol, as well as on other occasions. Finally, I would like to thank Sam Streeter, Tim Santens and Dan Loughran for discussions on the leading constant.  
	\section{Notation}
	\subsection{Natural numbers}
	We use the convention that the set of natural numbers $\mathbb{N}$ contains $0$ and we write $\mathbb{N}^*$ for the set of nonzero natural numbers. We also define the set of extended natural numbers $\lineN:=\mathbb{N}\sqcup \{\infty\}$.
	
	\subsection{Algebra and analysis}
	We typically denote vectors using boldface and write their components using a normal face together with an index. For example, we may write $\mathbf{s}=(s_1,\dots,s_n)$ for a vector in $\mathbb{R}^n$. For two vectors $\mathbf{a}, \mathbf{b}\in \mathbb{R}^n$, we write $\mathbf{a}>\mathbf{b}$ if $a_i>b_i$ for all $i=1,\dots,n$. We also denote the $i$-th basis vector of $\mathbb{R}^n$ by $\mathbf{e}_i$.
	
	For an abelian group $G$, we write $G_{\mathbb{Q}}=G\otimes_{\mathbb{Z}} \mathbb{Q}$ and $G_{\mathbb{R}}=G\otimes_{\mathbb{Z}} \mathbb{R}$ for its base change to $\mathbb{Q}$ and $\mathbb{R}$, respectively. We denote its dual by $G^\vee=\Hom(G,\mathbb{Z})$. For a symbol $D$ we write $\mathbb{Z}(D)$ for the group isomorphic to $\mathbb{Z}$ with generator $D$, and similarly we write $\mathbb{Q}(D)\cong \mathbb{Q}$ for the vector space with generator $D$. If $A$ is a set on which a group $G$ acts, then we denote its set of fixed points by $A^G$.
	
	The logarithm $\log$ refers to the natural logarithm.

	\subsection{Number fields and local fields} \label{section: PF fields}
	For a number field $K$, we denote its set of places by $\Omega_K$ and its set of finite places by $\Omega_K^{<\infty}$. For a place $v\in \Omega_K$, we denote by $K_v$ the completion of $K$ with respect to the absolute value $|\cdot|_v$.
	For a finite place $v\in \Omega_K$, we set $\cO_v$ to be the ring of $v$-adic integers
	$$\cO_v=\{x\in K_v\mid v(x)\geq 0\}.$$ For an infinite place $v$, we simply set $\cO_v=K_v$.
	
	Furthermore, for a finite set $S\subset \Omega_K$ containing the infinite places, we set $\cO_S$ to be the ring of all $x\in K$ which are integral with respect to all places not in $S$. In particular, if $S$ is the set of infinite places, then $\cO_K:=\cO_S$ is the ring of integers of the number field.
	\subsection{Geometry}
For a field $k$ we write $\overline{k}$ for a choice of an algebraic closure.
For a scheme $X$ over a base scheme $S$ and a morphism $S'\rightarrow S$ of schemes, we denote the base change by $S'$ as $X_{S'}:=X\times_S S'$. If $S'=\spec R$, we also write $X_{R}$ in place of $X_{S'}$.

If $D$ is an effective Cartier divisor on $X$, then we will routinely identify it with the closed subscheme of $X$ defined by the sheaf of ideals $\cO_X(-D)\subset \cO_X$.

We define a variety over a field $k$ to be a separated geometrically integral scheme of finite type over $k$, and a curve to be a variety of dimension $1$.

	
	\section{Pairs and \texorpdfstring{$\cM$}{M}-points} \label{section: pairs and M-points}
	In this section we introduce pairs and the $\cM$-points on such pairs. This notion will be based on the notion of pair previously considered in \cite{Moe24, Thesis}. In this article, we work with a more general notion. The differences are are as follows:
	\begin{enumerate}
		\item The parameter set $M$ does not only describe which intersection multiplicities $\cM$-points should have, it also can impose to which intersection strata these multiplicities should correspond.
		\item The divisors $D_1,\dots,D_n$ used to define the pair only need to be defined over a Galois extension of the base field, rather than over the base field itself. In this case, we use Galois descent to define $\cM$-points and the geometric invariants of the pair. 
	\end{enumerate}
	The pairs considered in \cite{Moe24, Thesis} are the pairs which are split and divisorial, as defined in this section. There are several reasons for considering more general pairs than these. For instance, on a non-split smooth toric variety, the torus-invariant prime divisors are not smooth in general, but their geometric components are smooth. This is significant, as our generalization of Manin's conjecture requires the divisors $D_1,\dots,D_n$ defining the pair to be smooth.
	\subsection{Pairs} \label{Section: split pairs}
	In this section we will introduce pairs and their $\cM$-points. We will introduce these notions in large generality, using the language of scheme theory. While we are mainly interested in $\cM$-points over rings of integers, we will also need to consider $\cM$-points over $\mathbb{P}^1$ in Definition \ref{def: rationally connected pair} to define what it means for a pair to be rationally connected. Furthermore, the general framework presented here has nice functorial properties as seen in Remark \ref{remark: (X,M) as a functor}.

	In general, we want to consider divisors which need not be not defined over the base, but only over a Galois cover. Let us recall the definition of a Galois cover as in \cite[Tag 03SF]{Stacks}, generalizing Galois extensions of a field.
	\begin{definition}
		Let $B$ and $A$ be connected schemes. Then a finite étale morphism $A\rightarrow B$ is a \textit{Galois cover} with group $G=\mathrm{Aut}_B(A)$ if the degree of $A\rightarrow B$ is equal to $\#G$.
	\end{definition}
	In this article, we will consider two types of Galois covers: Galois extensions of the ground field and Galois extensions of the ring of $S$-integers. If $E/K$ is an Galois extension of number fields and $S\subset \Omega_K$ is a finite collection of places including the infinite places as well as the places which ramify in the extension, then both $\spec E\rightarrow \spec K$ as well as $\spec \cO_{\tilde{S}}\rightarrow \spec \cO_{S}$ are Galois covers, where $\tilde{S}\subset \Omega_E$ is the set of places in $E$ above $S$.
\begin{assumption}
	For the remainder of the section, all schemes will be Noetherian schemes over $B$, and we let $A\rightarrow B$ be a fixed Galois cover with Galois group $G$.
\end{assumption}

	$\cM$-points are defined using local intersection multiplicities, which we introduce now.
	\begin{definition}
		Let $P\colon Y\rightarrow X$ be a morphism of schemes and let $D$ be an effective divisor on $X_{A}$. For any prime (Cartier) divisor $v$ on $Y_{A}$, let $\cO_v$ be the local ring of $Y_{A}$ at the generic point of $v$, and let $\pi_v$ be an uniformizer. Let $P_v\colon \spec \cO_v\rightarrow X$ be the morphism obtained by composing $P$ with the natural morphism $\spec \cO_v\rightarrow Y$.
		We define the \textit{(local) intersection multiplicity} of $P$ and $D$ at $v$ to be the extended natural number $n_v(D,P)\in \lineN$ such that the fiber product
		$$\begin{tikzcd}
			\spec \cO_v\times_X D \arrow[rr] \arrow[d, hook]    &                  & D \arrow[d, hook] \\
			\spec \cO_v \arrow[r] \arrow[rr, "P_v", bend left] & Y \arrow[r, "P"] & X                
		\end{tikzcd}$$
		is given as
		 $\spec \cO_v\times_X D=\spec \cO_v/(\pi_v)^{n_v(D,P)}$, where we set $(\pi_v)^\infty=0$.
		Similarly, for an irreducible component $v$ of $Y_A$, we set
		$$n_v(D,P)=\begin{cases}
			0 & \text{ if }P(v)\not\in D
			\\ \infty & \text{ if } P(v)\in D
		\end{cases}$$
	\end{definition}
	\begin{remark}
		The reason for the symbol $v$ for the divisor on $Y_A$ is because we often apply the definition when $X$ is the integral model of some variety over a number field $K$ and $A=Y=\spec \cO_K$. In this case $P$ is an integral point on $X$, $v$ corresponds to a prime ideal $\mathfrak{p}$ of $\cO_K$, and $n=n_v(D,P)$ is the maximal integer such that $P$ reduces to $D$ modulo $\mathfrak{p}^n$ (or $\infty$ if $P$ already lies on $D$).
	\end{remark}
	Note that if the image of $P$ is contained in $D$, then $n_v(D,P)=\infty$ for every Cartier divisor $v$ on $Y$. If $D$ is Cartier, the image of $P$ is not in $D$ and $Y$ is integral, then $P^{-1}D$ is the pullback divisor $P^*D$ \cite[Tag 02OO]{Stacks} and $n_v(D,P)$ is the coefficient of $v$ in $P^*D$.
	\begin{remark}
		Our definition of intersection multiplicity is an extension of the classical notion of local intersection multiplicity on varieties. If $X$ is a smooth surface over an algebraically closed field and $P\colon C\hookrightarrow X$ is a smooth curve not contained in a divisor $D\subset X$, then $$C\cap D= \sum_{\text{closed points }v\in C}n_v(D,P)v.$$ Thus $n_v(D,P)$ is the local intersection multiplicity of $C$ and $D$ in a point $v$ as defined in \cite[Chapter V]{Har77}.
	\end{remark}

Let us consider intersection multiplicities in a concrete example.
	\begin{example} \label{example: two intersections}
		If $X=\mathbb{A}_{\mathbb{Z}}^2$, $D_1=\{y=0\}$, and $D_2=\{y=x^2-1\}$, then an integral point $P=(a,b)\in \mathbb{Z}^2$ satisfies $n_p(D_1,P)=v_p(b)$ and $n_p(D_2,P)=v_p(b-a^2+1)$ for every prime number $p$, where $v_p$ is the $p$-adic valuation. In particular, it follows that $n_p(D_1,P), n_p(D_2,P)>0$ if and only if $b=0\bmod p$ and $a=\pm 1 \bmod p$. The intersection multiplicities at $p$ tells us whether the point reduces to $D_1\cap D_2=\{(1,0),(-1,0)\}$ modulo $p$. However, it does not tell us to which of the two points it reduces to. In order to keep track of this, we introduce the following notation.
	\end{example}
	\begin{definition} \label{definition: components}
		Let $X$ be a scheme, and let $\{D_1,\dots,D_n\}$ be a set of distinct effective divisors on $X_A$ which is closed under the action of $G$ on divisors. For $\bfm\in \lineN^n$, let $\cC_{\bfm}$ be the set of irreducible components of $\bigcap_{\substack{i=1 \\ m_i>0}}^n D_i$. For any $\bfm\in \lineN^n$, we call $c\in \cC_\bfm$ a \textit{stratum}.

We write $\lineN^n_{\cC}$ for the set of tuples $(\bfm,c)$, where $\bfm\in \lineN^n$ and $c\in \cC_{\bfm}$. The action of $G$ on $\{D_1,\dots,D_n\}$ induces an action on $\lineN^n$ by permuting the coordinates, and thus $G$ acts on $\lineN^n_{\cC}$ by $\sigma(\bfm,c)=(\sigma(\bfm),\sigma(c))$ for all $\sigma\in G$.

If for $\bfm\in \lineN^n$ there is a unique component $c\in \cC_{\bfm}$, then we routinely identify $\bfm$ and $(\bfm,c)$. Finally, for a subset $\fM\subset \lineN^n$, we write $\fM_{\red}$ for the set of all $\bfm\in \fM$ such that $\cC_{\bfm}\neq \emptyset$.

	\end{definition}

The strata $c$ in Definition \ref{definition: components} correspond to the faces of the corresponding Clemens complex (see e.g. \cite{San23-2}) whenever $X$ is a variety and $\sum_{i=1}^n D_i$ is a normal crossings divisor.
	
	Now we will introduce the augmented multiplicity, which keeps track of both the intersection multiplicities as well as to which strata the point reduces to.
	\begin{definition}
		Let $X$ be a scheme and let $D_1,\dots,D_n$ be a finite tuple of effective Cartier divisors on $X_A$. For any morphism $P\colon Y\rightarrow X$ and any prime (Cartier) divisor or irreducible component $v\subset Y_A$ we define the \textit{multiplicity} of $P$ at $v$ to be
		$$\mult_v(P):=(n_v(D_1,P),\dots, n_v(D_n,P))\in \lineN^n,$$
		and the \textit{augmented multiplicity} to be
		$$\mult^\cC_v(P):=\{(\mult_v(P), c)\in \lineN^n_{\cC}\mid P(v)\in c\}.$$
	\end{definition}
	
	\begin{remark}
		For many choices of divisors $D_1,\dots,D_n$, the intersection $\bigcap_{i\in I} D_i$ is irreducible or empty for every $I\subset \{1,\dots,n\}$. In this case, the multiplicity of a point $P$ at a divisor $v$ completely determines the augmented multiplicity. For instance, this property is satisfied in the following example.
	\end{remark}
	
	\begin{example} \label{example: multiplicity projective space}
		Let $X=\P^{n-1}_{\mathbb{Z}}$ and let $D_i$ be the $i$-th coordinate hyperplane. For every prime number $p$, the multiplicity is given on $p$-adic points on $X$ by $\mult_p(a_1:\dots: a_n)= (v_p(a_1),\dots,v_p(a_n))$, provided we normalize the coordinates such that $a_1,\dots,a_n\in \mathbb{Z}_p$ and such that at least one coordinate has valuation $1$.
	\end{example}
	The previous example has a straightforward generalization to toric varieties, see e.g. \cite[\S 5]{Moe24}.
	
	\begin{example}
		Let us compute $\mult^\cC_v(P)$ in Example \ref{example: two intersections} for integral points $P=(a,b)\in \mathbb{Z}^2$. For every prime number $p$ we write $\bfm_p=(v_p(b),v_p(b-a^2+1))$, and we will compute the augmented multiplicity of $P$. We consider several cases based on the multiplicity:
		\begin{itemize}
			\item If $v_p(b)=v_p(b-a^2+1)=0$, then $\mult_p^{\cC}(a,b)=\{(\bfm_p, X)\}$.
			\item If $v_p(b)>0$ and $v_p(b-a^2+1)=0$, then $\mult_p^{\cC}(a,b)=\{(\bfm_p, D_1)\}$,
			\item If $v_p(b)=0$ and $v_p(b-a^2+1)>0$, then $\mult_p^{\cC}(a,b)=\{(\bfm_p, D_2)\}$,
			\item If $v_p(b),v_p(b-a^2+1)>0$ and $p$ is odd, then
			$$\mult_p^{\cC}(a,b)=\big\{\big(\bfm_p, \{(\overline{a},0)\}\big)\big\},$$
			where $\overline{a}\in \{1,-1\}$ is the reduction of $a$ modulo $p$.
			\item If $v_p(b),v_p(b-a^2+1)>0$ and $p=2$, then $$\mult_2^{\cC}(a,b)=\big\{\big(\bfm_2, \{(1,0)\}\big), \big(\bfm_2, \{(-1,0)\}\big)\big\}.$$
		\end{itemize}
		The reason why both components of $D_1\cap D_2$ appear in the last case is because $-1=1$ in $\mathbb{F}_2$, so $(D_1\cap D_2)_{\mathbf{F}_2}=\{(1,0)\}$.
	\end{example}
	The previous example also gives an example of two divisors $D_1,D_2$ such that the intersection $D_1\cap D_2$ is connected, but consists of two irreducible components.
	
	Now we are ready to define split pairs and $M$-points on them.
	\begin{definition} \label{def: split pair}
		Let $X$ be a scheme over a base scheme $B$, and let $D_1,\dots,D_n$ be effective divisors on $X_A$ as in Definition \ref{definition: components}. Let $\fM_{\cC}\subset \lineN^n_{\cC}$ be a subset closed under the action of $G$ containing $(\mathbf{0},X)$ and write $M=((D_1,\dots,D_n), \fM_{\cC})$. We call $(X,M)$ a \textit{pair} over $B$. Furthermore, we call $(X,M)$ \textit{split} if $G=1$ (so $A=B$).
		
		Let $Y$ be a regular, Noetherian scheme over $B$. We call a morphism $P\colon Y\rightarrow X$ of schemes over $B$ an \textit{$M$-point over $Y$} if $\mult^{\cC}_v(P)\subset \fM_{\cC}$ for every prime divisor and irreducible component $v\subset Y_A$. We denote the set of $M$-points over $Y$ by $(X,M)(Y)$.
Two pairs $(X,M)$ and $(X,M')$ are \textit{equivalent} if they have the same set of $M$-points for all such schemes $Y$. Similarly, for an open subscheme $X'\subset X$, we write $(X',M')\subset (X,M)$ if the inclusion holds for all such $Y$.
		If $A\rightarrow B$ corresponds to a Galois extension $E/K$ of fields, then we say that $E$ is the \textit{splitting field} of $(X,M)$.
	\end{definition}
The field $E$ in Definition \ref{def: split pair} is called as such since the base change $(X_E,M)$ is a split pair over $E$.
	\begin{remark}
		If $Y$ is irreducible, then the condition on the irreducible components $v$ of $Y_A$ is automatically satisfied if the image of $Y_A$ is not contained in any of the divisors.
	\end{remark}
	\begin{remark}
		The assumption that $Y$ is regular and Noetherian implies that every divisor on $Y$ is a sum of prime Cartier divisors by \cite[Tag 0BCP]{Stacks}. Therefore, the values $\mult_v(P)$ for all $v$ as above together determine the inverse images $P^{-1}D_1,\dots,P^{-1}D_n\subset Y_A$.
	\end{remark}
	
	For many pairs considered in this article, no condition is imposed on the strata $c$ but only on the intersection multiplicity. To make the notation less cumbersome in this case, we introduce the following definition.
	\begin{definition} \label{def: divisorial pair}
		Let $X$ be a scheme and let $D_1,\dots,D_n$ be effective divisors on $X_A$ as in Definition \ref{def: split pair}. For a set $\fM\subset \lineN^n$ we identify $(X,M)$ where $M=((D_1,\dots,D_n), \fM)$ with the pair $(X,M')$, where $M'=((D_1,\dots,D_n), \fM_{\cC})$ and $$\fM_{\cC}=\{(\bfm,c)\mid \bfm\in \fM,\, c\in \cC_{\bfm} \}.$$
		We call such a split pair \textit{divisorial} and we call the set $\fM$ a \textit{set of multiplicities} for $(X,M)$.
		In particular, it follows that a morphism $P\colon Y\rightarrow X$ is an $M$-point if $\mult_v(P)\in \fM$ for every prime divisor and irreducible component $v\subset Y_A$.
	\end{definition}
	The term ``divisorial'' is used, because in this case the criterion for a morphism $P\colon Y\rightarrow X$ to be an $M$-point depends solely on the intersection multiplicities of $P$ with the divisors $D_1,\dots,D_n$, and not on their intersection strata.
	\begin{remark}
		The definition of a divisorial split pair given here strongly resembles the definition of a pair previously used by the author in \cite[Definition 3.1]{Moe24}, but there are two minor differences. In this article, we assume for simplicity that $D_1,\dots,D_n$ are divisors rather than arbitrary closed subschemes. Furthermore, in the other article, we imposed a weak condition on the set $\fM$ ensuring that the $\cM$-points over $\mathbb{Z}$ are a subset of the $\cM$-points over $\mathbb{Q}$. This condition is not necessary in the framework used in this paper, as the $\cM$-point condition used here also takes generic points into account.
	\end{remark}
	The augmented multiplicity $\mult_v^\cC$ can frequently be regarded as a function to $\lineN^n_{\cC}$, as the next remark illustrates. 
	\begin{remark}
		If the divisors $D_1,\dots, D_n$ are chosen such that $D_1+\dots+D_n$ is a strict normal crossings divisor as in \cite[Tag 0BI9]{Stacks}, then \cite[Tag 0BIA]{Stacks} implies that the connected components of the intersections $\bigcap_{i\in I} D_i$ are all irreducible for all nonempty $I\subset \{1,\dots, n\}$. Thus it follows that for any scheme $Y$ over $B$ with a prime divisor (or connected component) $v$, the function $\mult_v^{\cC}$ can be regarded as a function $X(Y)\rightarrow \lineN^n_{\cC}$. In particular, the condition at $v$ for $P\colon Y\rightarrow X$ to be an $M$-point simplifies to $\mult^{\cC}_v(P)\in \fM_{\cC}$.
	\end{remark}
	For any pair $(X,M)$, we introduce a bigger pair $(X,M_{\mon})$, which behaves well with respect to morphisms of schemes.
	\begin{notation} \label{notation: M_mon}
		For a pair $(X,M)$, we let $(X,M_{\mon})$ be the pair $(X,((D_1,\dots,D_n),\fM_{\cC,\mon}))$, where $\fM_{\cC,\mon}\subset \lineN^n_{\cC}$ is the smallest set containing $\fM_{\cC}$ such that for every $(\bfm,c),(\bfm',c')\in \fM_{\cC,\mon}$ we have $(\bfm+\bfm',\tilde{c})\in \fM_{\cC,\mon}$ for all irreducible components $\tilde{c}$ of $c\cap c'$, and furthermore $(\infty \bfm,c)\in \fM_{\cC,\mon}$ (where we use the convention that $\infty \cdot 0=0$). If $(X,M)$ is divisorial with a set of multiplicities $\fM$, then $(X,M_{\mon})$ is also divisorial and the closure of the monoid $\fM_{\mon}$ generated by $\fM$ in $\lineN^n$ is a set of multiplicities for this pair.
	\end{notation}
	\begin{remark} \label{remark: (X,M) as a functor}
		Suppose $(X,M)$ is a pair over a scheme $B$ such that the divisors $D_1,\dots,D_n$ are Cartier divisors. Then the assignment $$Y\mapsto (X,M_{\mon})(Y)$$
		is a functor $(X,M)$ from the category of regular Noetherian schemes over $B$ to the category of sets. This follows from the following observation: if $f\colon Y'\rightarrow Y$ is a morphism of such schemes over $S$ and $P\colon Y\rightarrow X$ is a morphism over $S$, then for every prime divisor $v'$ on $Y'_A$, we have $$\mult_{v'}(P\circ f)= \sum_{v \text{ prime divisor on }Y_A} a_v \mult_{v}(P),$$
		for $a_v\in \lineN$. Here $a_v$ is the multiplicity of $v'$ in the pullback of $v$ to $Y'_A$ if $f(Y')\not\subset v$ and $a_v=\infty$ otherwise. This can be seen as follows: if the image of $P\circ f$ is not contained in an effective Cartier divisor $D\subset X$, then the scheme-theoretic fiber $(P\circ f)^{-1} D=(P\circ f)^* D$ is a Cartier divisor on $Y'$. The desired identity now follows from $(P\circ f)^* D=f^*P^* D$ together with the linearity of $f^*$.
	\end{remark}

\subsection{Pairs over number fields}
	In this article, we are primarily interested in pairs over number fields, and their integral models, which we will now define.
	\begin{definition} \label{def: integral model split pair}
		Let $K$ be a number field with a Galois extension $E$, let $S\subset \Omega_K$ be a finite set of places including the infinite places and the places ramifying in $E$, and let $\tilde{S}$ be the places in $E$ above $S$. Let $X$ be a proper variety over $K$ and let $(X,M)$ be a pair over $K$ with splitting field $E$. A proper scheme $\cX$ over $\cO_S$ together with an isomorphism $\cX_K\cong X$ is called an \textit{integral model} of $X$. For an integral model $\cX$, a pair $(\cX,\cM)=(\cX,((\cD_1,\dots,\cD_n), \fM'_{\cC}))$ with $\cD_1,\cD_n\subset \cX_{\cO_{\tilde{S}}}$ is called an \textit{integral model} of $(X,M)$ if $\cD_{1,E}= D_1,\dots,\cD_{n,E}=D_n$ and $$\fM_{\cC}=\{(\bfm,c)\mid (\bfm,c_K)\in \fM_{\cC}', c_K\neq \emptyset\}.$$
		(Here we used the isomorphism to identify the subschemes of $\cX_K$ with those of $X$.)
	\end{definition}
	\begin{remark}
		Let $(X,M)$ be a pair over a number field $K$, where $X$ is a proper variety. For an integral model $\cX$ of $X$ over $\cO_S$, where $S$ is a finite set of places of $K$, there is a natural choice of an $\cO_S$-integral model $(\cX,\cM)=(\cX,((\cD_1,\dots,\cD_n), \fM'_{\cC}))$ of $(X,M)$. This is given by taking the divisors $\cD_1,\dots,\cD_n$ to be the Zariski closures of the divisors $D_1,\dots,D_n$ in $\cX_{\cO_{\tilde{S}}}$ and by taking $\fM'_{\cC}$ to be the set obtained by replacing the strata $c$ in $\fM_{\cC}$ with their closures $\overline{c}$.
		(Here we take the Zariski closure of a not necessarily reduced scheme to mean the closure as defined in \cite[Proposition 3.6]{Moe24}.)
		
		Despite having this canonical choice of an integral model, it is useful to consider other integral models. This discussed in \cite[\S 3.5]{Moe24}, of which the main takeaway is as follows. If $f\colon \cY\rightarrow \cX$ is a morphism of $\cO_S$-integral models, then one can pull back $\cM$ to get a pair $(\cY,f^{-1}\cM)$ whose points are exactly the preimages of the points on $(\cX,\cM)$, but the pair $(\cY,f^{-1}\cM)$ is not necessarily the canonical integral model for $(Y,f^{-1}M)$.
	\end{remark}
The next proposition implies that a pair over a field $K$ only depends on the divisors $D_{1,\overline{K}},\dots,D_{n,\overline{K}}$ and $\fM_{\cC}$, rather than on the particular Galois extension $E/K$ used to define the pair.
	\begin{proposition}
		Let $\widetilde{B}\rightarrow B$ be a Galois cover of schemes and let $P\colon Y\rightarrow X$ be a morphism of regular schemes over $B$. For any effective divisor $D$ on $X$ and any prime divisor $v$ on $Y$ we have
		$$n_v(D,P)=n_{\tilde{v}}(D_{\widetilde{B}},P_{\widetilde{B}}),$$
		for every prime divisor $\tilde{v}$ contained in the divisor $v\times_B \tilde{B}$.
		
	\end{proposition}
	\begin{proof}
		Without loss of generality, we can assume that $Y=\spec \cO_v$ is a discrete valuation ring. If $n_v(D,P)=\infty$, then the image of $v$ in $X$ lies in $D$, so the image of $\tilde{v}$ lies in $D_{\widetilde{B}}$ so $n_{\tilde{v}}(D_{\widetilde{B}},P_{\widetilde{B}})=\infty$. Thus we can assume that the image of $P$ is not contained in $D$. Then the fibers $P^{-1}D$ and $P_{\widetilde{B}}^{-1} D_{\widetilde{B}}$ are divisors on $Y$ and $Y_{\widetilde{B}}$, and $n_v(D,P)$ and $n_{\tilde{v}}(D_{\widetilde{B}},P_{\widetilde{B}})$ are the coefficients of $v$ and $\tilde{v}$ in these divisors. Since $P_{\widetilde{B}}^{-1} D_{\widetilde{B}}$ is the pullback of $P^{-1}D$ under the étale morphism $Y_{\widetilde{B}}\rightarrow Y$, this implies that $n_v(D,P)=n_{\tilde{v}}(D_{\widetilde{B}},P_{\widetilde{B}})$.
	\end{proof}
	\subsection{Examples of \texorpdfstring{$\cM$}{M}-points} \label{section: examples M-points}
	In this section we give some examples of $\cM$-points, such as integral points, Campana points and Darmon points, as well as their geometric counterparts. For more examples, see \cite[\S 3.3]{Moe24}. We fix a proper variety $X$ over a number field $K$ with an integral model $\cX$ over $\cO_S$ for some finite set of places $S\subset \Omega_K$, and we fix an $\cO_S$-scheme $Y$ (such as the spectrum of $\cO_S$ or $\cO_v$ for a place $v\in \Omega_K\setminus S$). On $X$, we choose prime divisors $D_1,\dots, D_n$ and we denote their Zariski closures in $\cX$ by $\cD_1,\dots, \cD_n$.
	\begin{example} \label{example: integral points}
		If $(\cX,\cM)$ is the pair given by $\fM=\{(0,\dots,0)\}$, then the $\cM$-points over $Y$ on the pair are exactly the integral points over $Y$ on $\cU=\cX\setminus \bigcup_{i=1}^n \cD_i$. More generally, if $(\cX,\cM)$ is the pair given by $\fM=\{\bfm\mid m_i=0\, \forall i\in I\}$ for $I\subset\{1,\dots n\}$, then the $\cM$-points on the pair are the integral points on $\cX\setminus \bigcup_{i\in I} \cD_i$.
	\end{example}
	
	\begin{definition} \label{def: Campana points and Darmon points}
		Let $m_1,\dots,m_n\in \mathbb{N}^*\cup\{\infty\}$ and define the $\mathbb{Q}$-divisors $D_{\bfm}=\sum_{i=1}^n\left(1-\frac{1}{m_i}\right)D_i$ and $\cD_{\bfm}=\sum_{i=1}^n\left(1-\frac{1}{m_i}\right)\cD_i$, where we set $\frac{1}{\infty}=0$. We call the tuples $(X,D_{\bfm})$ and $(\cX,\cD_{\bfm})$ \textit{Campana pairs}.
		There are several sets of $\cM$-points we can associate to a Campana pair:
		\begin{itemize}
			\item \textit{Darmon points on $(\cX,\cD_{\mathbf{m}})$} over $Y$ are the $\cM$-points over $Y$ for $(\cX,\cM)$, where $\fM$ is the collection of $\bf{w}\in \lineN^n$ for which $m_i|w_i$ for all $i\in \{1,\dots,n\}$. Here we use the conventions that the only integer divisible by $\infty$ is $0$ and that $\infty$ is divisible by every positive integer.
			\item \textit{Campana points on $(\cX,\cD_{\mathbf{m}})$} over $Y$ are the $\cM$-points over $Y$ for the pair $(\cX,\cM)$, where $\fM\subset \lineN^n$ is the collection of $\bf{w}\in \lineN^n$ such that for all $i\in \{1,\dots,n\}$ we have
			\begin{enumerate}
				\item $w_i=0$ if $m_{i}=\infty$ and
				\item $w_i=0$ or $w_i\geq m_i$ if $m_i\neq \infty$.
			\end{enumerate}
			\item \textit{Weak Campana points on $(\cX,\cD_{\mathbf{m}})$} over $Y$ are the $\cM$-points over $Y$ for $(\cX, \cM)$, where $\fM$ is the collection of all $\bf{w}\in \lineN^n$ such that
			\begin{enumerate}
				\item $w_i=0$ if $m_{i}=\infty$ and
				\item either $w_i=0$ for all $i\in \{1,\dots,n\}$ or $$\sum_{\substack{i=1 \\ m_{i}\neq 1}}^n \frac{w_i}{m_i}\geq 1.$$
			\end{enumerate}
		\end{itemize}
	\end{definition}
	Note that for $I\subset \{1,\dots,n\}$, the Darmon points and the (weak) Campana points on $(\cX,\sum_{i\in I}\cD_i)$ are just the integral points on $\cX\setminus \bigcup_{i\in I} \cD_i$ as in Example \ref{example: integral points}.
	
	\begin{example}
		To illustrate $\cM$-points and the above definitions, let us take $X=\mathbb{P}^{n-1}_{\mathbb{Q}}$ with the integral model $\cX=\mathbb{P}^{n-1}_{\mathbb{Z}}$ over $\mathbb{Z}$. Furthermore, we choose the divisors $D_1=\{X_1=0\},\dots,D_n=\{X_n=0\}$ to be the coordinate hyperplanes in $\mathbb{P}^{n-1}_{\mathbb{Q}}$ and we take $\cD_1,\dots,\cD_n$ to be their Zariski closures. Using the description of the multiplicity map from Example \ref{example: multiplicity projective space}, we give examples of $\cM$-points on projective space. In what follows, let $(a_1:\dots:a_n)$ be a rational point, written such that $(a_1,\dots,a_n)\in \mathbb{Z}^n$ is a tuple of coprime integers.
		
		Let $m_1,\dots,m_n\in \mathbb{N}^*\cup \{\infty\}$ and set $\cD_{\mathbf{m}}= \sum_{i=1}^n\left(1-\frac{1}{m_i}\right)\cD_i$ as before.
		\begin{itemize}
			\item A point $(a_1:\dots: a_n)$ is a Darmon point over $\mathbb{Z}$ on $(\cX, \cD_{\bfm})$ if and only if $|a_i|$ is an $m_i$-th power for all $i\in \{1,\dots, n\}$, or $|a_i|=1$ if $m_i=\infty$. The absolute value sign here accounts for the fact that $\mathbb{Z}^*=\{1,-1\}$.
			\item A point $(a_1:\dots: a_n)$ is a Campana point over $\mathbb{Z}$ on $(\cX, \cD_{\bfm})$ if and only if $a_i$ is an $m_i$-full number for all $i\in \{1,\dots, n\}$. Here we recall that a number $a\in \mathbb{Z}$ is $m$-full, for a positive integer $m$, if for every prime $p$ dividing $a$, $p^m$ also divides $a$. Furthermore, we use the convention that the only $\infty$-full integers are $1$ and $-1$.
			\item For simplicity, assume that $m_1=\dots=m_n=m$ for some positive integer $m$. A point $(a_1:\dots: a_n)$ is a weak Campana point over $\mathbb{Z}$ on $(\cX, \cD_{\bfm})$ if and only if the product $a_1\dots a_n$ is an $m$-full number.
		\end{itemize}
\end{example}
		\begin{remark} \label{remark: Darmon points and root stacks}
			As explained in \cite[\S 8]{Moe24}, Darmon points can be viewed as integral points on the root stack $(\cX,\sqrt[\bfm]{\cD})$ corresponding to the Campana pair $(\cX,\cD_{\bfm})$.
		\end{remark}

Geometric Campana points and geometric Darmon points are defined similarly, but use the geometric components of $D_{\bfm}$ rather than the components over $K$.
\begin{definition}
		Let $(X,D_{\bfm})$ be a Campana pair over a number field $K$ such that the support of $D_{\bfm, \overline{K}}$ is a strict normal crossings divisor. Let $E/K$ be a Galois extension such that we can write $D_{\bfm,E}=\sum_{i=1}^n \left(1-\frac{1}{m_i} \right)D_i$ for smooth divisors $D_1,\dots, D_n$ on $X_E$. Let $S\subset \Omega_K$ be a finite set of places containing the infinite places and the primes which ramify in $E/K$ and let $\tilde{S}\subset \Omega_E$ be the set of places above $S$.
		Let $(\cX,\cD_{\bfm})$ be an integral model of $(X,D_{\bfm})$ over $\cO_S$ as in Definition \ref{def: Campana points and Darmon points} and write $\cD_{\bfm,\cO_{\tilde{S}}}=\sum_{i=1}^n \left(1-\frac{1}{m_i} \right)\cD_i$, where $\cD_1,\dots,\cD_n$ are the Zariski closures of the divisors $D_1,\dots, D_n$ in $\cX_{\cO_{\tilde{S}}}$.
		Let $Y$ be a scheme over $\cO_S$ (such as $\spec \cO_S$).
		\begin{itemize}
			\item \textit{Geometric Darmon points on $(\cX,\cD_{\bfm})$} over $Y$ are the $\cM$-points over $Y$ for $(\cX,\cM)$, where $\fM$ is the collection of $\mathbf{w}\in \lineN^n$ for which $m_i| w_i$ for all $i\in \{1,\dots,n\}$.
			\item \textit{Geometric Campana points on $(\cX,\cD_{\bfm})$} over $Y$ are the $\cM$-points over $Y$ for $(\cX,\cM)$, where $\fM$ is the collection of $\mathbf{w}\in \lineN^n$ such that for all $i\in \{1,\dots,n\}$ we have
			\begin{enumerate}
				\item $w_i=0$ if $m_{i}=\infty$ and
				\item $w_i=0$ or $w_i\geq m_i$ if $m_i\neq \infty$.
			\end{enumerate}
		\end{itemize}
	\end{definition}
	Geometric Campana points can be viewed as the subset of Campana points which remain Campana points after enlarging the base field. In particular, geometric Campana points are just the same as Campana points whenever the support of the divisor $D_{\bfm}$ is a strict normal crossings divisor (over the ground field $K$). For Darmon points analogous statements hold.
	\begin{remark}
		In this article, we only consider geometric Campana points and geometric Darmon points over rings of integers $\cO_S$ which do not ramify in the splitting field $E/K$. In \cite{ShSt24}, geometric Campana points and geometric Darmon points of more general rings of integers were considered. In their setting, Campana points and Darmon points can differ from their geometric counterparts due to ramification, see \cite[Lemma 2.7]{ShSt24}.
	\end{remark}
	
Let us consider an example where geometric Campana points differ from Campana points.
\begin{example} \label{Example: Campana points not preserved under base change}
		Consider the Campana pair $(\mathbb{P}_{\mathbb{Q}}^2, \tfrac{1}{2}D)$ with $D=\{X^2+Y^2=0\}$. Over $\mathbb{Q}(\sqrt{-1})$ the divisor $D$ splits up as
		$$D_{\mathbb{Q}(\sqrt{-1})}=D_1+D_2,$$
		where $D_1=\{X+\sqrt{-1}Y=0\}$ and $D_2=\{X-\sqrt{-1}Y=0\}$. Let $\cD$ be the Zariski closure of $D$ in $\mathbb{P}^2_{\mathbb{Z}}$ and similarly let $\cD_1$ and $\cD_2$ be the Zariski closures of $D_1$ and $D_2$ in $\mathbb{P}^2_{\mathbb{Z}[i]}$. The point $(14:2:1)$ is a Campana point over $\mathbb{Z}$ on the pair $(\mathbb{P}^2_\mathbb{Z},\tfrac{1}{2}\cD)$ as $14^2+2^2=2^3\cdot 5^2$. Nevertheless, for the place $v=(1-2\sqrt{-1})$ we have $$n_v(\cD_1,(14:2:1))=v(14+2\sqrt{-1})=1$$ as $14+2\sqrt{-1}=2(1-2\sqrt{-1})(1+3\sqrt{-1})$. This implies that $(14:2:1)$ is not a $v$-adic Campana point on the pair $(\mathbb{P}^2_{\mathbb{Z}[\sqrt{-1}]},\tfrac{1}{2}\cD_1+\tfrac{1}{2}\cD_2)$, so $(14:2:1)$ is not a geometric Campana point on $(\mathbb{P}_{\mathbb{Q}}^2, \tfrac{1}{2}D)$.
	\end{example}

Besides Campana points, Darmon points, and (weak) Campana points, there are many more interesting sets of $\cM$-points, see \cite[\S 3.3]{Moe24}.

	\section{The Picard group of a pair and its geometry}
	In this section we will introduce the Picard group and the canonical divisor class of smooth pairs over a field $K$, and we will use this to define the Fujita and $b$-invariants used in the formulation of Conjecture \ref{conj: generalized Manin conjecture}.
	
	For pairs $(X,M)$ corresponding to Darmon points, there is a natural definition for the Picard group of $(X,M)$. For such a pair, we have a corresponding root stack $(X,\sqrt[\bfm]{D})$ which lies over $X$. In this case we would like the Picard group of $(X,M)$ to be the Picard group of $(X,\sqrt[\bfm]{D})$, and the canonical divisor class to be the canonical divisor class on the stack. Both the Picard group and the canonical divisor class can be computed using the natural morphism $(X,\sqrt[\bfm]{D})\rightarrow X$, which is an isomorphism over the open set $U=X\setminus (D_1\cup\dots \cup D_n)$ and ramified over the divisors $D_i$ with multiplicity $m_i$.
	
	More generally, we can regard $(X,M)$ as a geometric space akin to a scheme or stack lying above $X$. If $(X,M)=(X,M_{\mon})$ then $(X,M)$ determines a functor $Y\mapsto (X,M)(Y)$ as we have seen in Remark \ref{remark: (X,M) as a functor}. We view the natural inclusion map $(X,M)\rightarrow X$ as analogous to a ``generically finite morphism" which is an isomorphism over $U$. Inspired by this view we introduce corresponding objects to the pair $(X,M)$, such as divisors and a Picard group.
	
	\subsection{Divisors on pairs}
	Before we can define divisors on pairs, we first need to define the generators of the pair, which we define similarly as for monoids. 
	\begin{definition}
		For a pair $(X,M)$ we define its \textit{set of generators} to be the set $\Gamma_{M,\cC}$ of all $(\bfm,c)\in \fM_{\cC}\cap \mathbb{N}^n_{\cC}$ such that the only subset $J\subset \fM_{\cC}\setminus\{(\mathbf{0},X)\}$ such that $$\bfm=\sum_{(\bfm',c')\in J}\bfm' \quad\text{  and  }\quad c\subset \bigcap_{(\bfm',c')\in J} c'$$ is $J=\{(\bfm,c)\}$.
		If $(X,M)$ is a divisorial pair defined by a set of multiplicities $\fM\subset \lineN^n$, then we set $$\Gamma_M=\{\bfm\in \fM_{\red}\cap \mathbb{N}^n\mid \bfm\text{ is not a sum of two nonzero elements in }\fM_{\mon}\}.$$
	\end{definition}
	If $(X,M)$ is a divisorial pair then the above definition implies that we have $$\Gamma_{M,\cC}=\{(\bfm,c)\mid \bfm\in \Gamma_M, c\in \cC_{\bfm}\}.$$
	\begin{example}
		If $\fM_{\cC,\fin}=\{(\mathbf{0},X)\}$, then $\Gamma_{M,\cC}$ is the empty set.
	\end{example}
	\begin{definition} \label{def: smooth pair}
		A pair $(X,M)$ over a field $K$ is called \emph{smooth} if
		\begin{enumerate}
			\item the set of generators $\Gamma_{M,\cC}$ is finite.
			\item $X$ is a smooth variety over $K$, 
			\item every divisor $D_i$ is connected, nonempty and smooth over $K$ while $\sum_{i=1}^n D_i$ is a strict normal crossings divisor as defined in \cite[Tag 0BI9]{Stacks},
		\end{enumerate}
	\end{definition}
	The use of the term ``smooth'' is motivated by the fact that strict normal crossings pairs as considered in logarithmic geometry are log smooth, see e.g. \cite[Chapter IV, Example 3.1.14]{Ogu18}. Furthermore, \cite[Proposition 8.3]{Moe24} shows that a root stack $(X, \sqrt[\bfm]{D})$ is smooth if its corresponding pair $(X,M)$ is smooth.

\begin{proposition} \label{prop: equivalent to smooth pair}
Let $(X,M)$ be a pair over a field $K$ and assume that $X$ is smooth and that $\#\Gamma_{M,\cC}<\infty$. If the divisor $\sum_{i=1}^n D_{i,\overline{K}}$ is a strict normal crossings divisor, then $(X,M)$ is equivalent to a smooth pair $(X,M')$ over $K$.
\end{proposition}
\begin{proof}
 The pair $(X,M')$ is constructed as follows. Let $E/K$ be a Galois extension over which all irreducible components of $D_1,\dots,D_n$ are defined. For $i\in \{1,\dots,n\}$, let $D_{i,E}= \sum_{\alpha\in \cA_i} D'_{i,\alpha}$ be the decomposition into irreducible components. Let $\fM'_{\cC}$ be the set of all $(\bfm',c')$, where $\bfm'\in \mathbb{N}^{\cA_1\times\dots\times \cA_n}$ and $c\in \cC_{\bfm'}$ such that $$\bigg(\bigg(\sum_{\alpha\in \cA_1} m_{1,\alpha}, \dots, \sum_{\alpha\in \cA_n} m_{n,\alpha}\bigg),c\bigg)\in \fM_{\cC},$$
		where $c'\subset c$. Then the pair $(X,M')$ over $K$ defined using this data is equivalent to $(X,M)$. Furthermore, by the assumptions on $(X,M)$ it follows that $(X,M')$ is smooth.
\end{proof}
	\begin{remark} \label{remark: equivalent to smooth pair Campana}
		Let $(X,D_{\bfm})$ be a Campana pair, where $X$ is smooth and the support of $D_{\bfm,\overline{K}}$ is a strict normal crossings divisor. As a special case of the previous proposition, we see that any pair $(X,M)$ corresponding to (weak) Campana or Darmon points on $(X,D_{\bfm})$ is equivalent to a smooth pair.
	\end{remark}
	
	\begin{assumption} \label{assumption: rationally connected characteristic 0}
		From now on, all pairs considered will be smooth. Furthermore, we will always assume that $X$ is a connected, proper variety over a field of characteristic $0$, unless mentioned otherwise.
	\end{assumption}

	For some results, we will need to assume that the pair is proper, as we will now define.
	\begin{definition} \label{def: proper pair}
		A pair $(X,M)$ is \textit{proper} if $X$ is proper and $\fM_{\cC}$ contains elements $(d_1\bfe_1,D_1),\dots, (d_n\bfe_n,D_n)$ for positive integers $d_1,\dots,d_n$.
	\end{definition}
	The set of generators of a proper pair is always finite, as the next proposition shows.
	\begin{proposition}
		Let $(X,M)$ be a proper pair over a field $K$. Then the set of generators $\Gamma_{M,\cC}$ is finite.
	\end{proposition}
	\begin{proof}
		Assume that $(d_1\bfe_1,D_1),\dots, (d_n\bfe_n,D_n)\in \fM_{\cC}$ for some positive integers $d_1,\dots,d_n$.
		For a stratum $c$, let $\fM_c$ (respectively $\Gamma_c$) be the set of all $\bfm\in \mathbb{N}^n$ such that $(\bfm,c')\in \fM_{\cC}$ (respectively $\in \Gamma_{M,\cC}$) for $c\subset c'$. Let $r\in \mathbb{Z}^n/(d_1\mathbb{Z}\times\dots\times d_n\mathbb{Z})$ be a nonzero residue class and let $\fM_{c,r}\subset \fM_c$ be the subset of elements which have residue class $r$. The elements $\Gamma_c \cap\fM_{c,r}$ are minimal elements in $\fM_{c,r}$ with respect to the partial order induced from $\mathbb{N}^n$. This implies that $\Gamma_c \cap\fM_{c,r}$ is finite so $\Gamma_c$ is finite. As there are only finitely many strata $c$, it follows that $\Gamma_{M,\cC}$ is finite.
	\end{proof}

	\begin{example} \label{example: generators Darmon and Campana}
		If $(X,M)$ is the pair corresponding to the Darmon points for the Campana pair $(X,D_{\bfm})$, where $D_{\bfm}=\sum_{i=1}^n\left(1- \frac{1}{m_i}\right)D_i$ for positive integers $m_1,\dots,m_n$, then $\Gamma_{M}=\{m_1\bfe_1,\dots, m_n\bfe_n\}$. If $(X,M')$ is the pair corresponding to the Campana points for the Campana pair $(X,D_{\bfm})$, then $\Gamma_{M'}=\{m_1 \bfe_1,\dots, (2m_1-1)\bfe_1, m_2\bfe_2,\dots, (2m_n-1)\bfe_n\}$.
	\end{example} 
	On a split pair $(X,M)$, we associate a formal symbol $\tilde{D}_{\bfm,c}$ to each element $(\bfm,c)\in \Gamma_{M,\cC}$, which we will refer to as a prime divisor on $(X,M)$.
	\begin{definition} \label{def: divisors (X,M)}
		Let $(X,M)$ be a smooth split pair over a field $K$. The \textit{group of divisors on $(X,M)$} is
		$$\Div(X,M):=\Div(U)\times\bigoplus_{(\bfm,c)\in \Gamma_{M,\cC}} \mathbb{Z}(\tilde{D}_{\bfm,c}),$$
		where $\Div(U)$ is the group of divisors on $U$.
		More generally, for a smooth pair $(X,M)$ over $K$ with splitting field $E$, the \textit{group of divisors on $(X,M)$} is
		$$\Div(X,M)=\Div(X_E,M)^G,$$
		where $G=\gal(E/K)$ acts by $\sigma\left(\tilde{D}_{\bfm,c}\right)=\tilde{D}_{\sigma(\bfm,c)}$.
		A \textit{divisor} on $(X,M)$ is an element of $\Div(X,M)$. Analogously, we define $\mathbb{Q}$-divisors and $\mathbb{R}$-divisors to be the elements in $\Div(X,M)_{\mathbb{Q}}$ and $\Div(X,M)_{\mathbb{R}}$, respectively.
		A divisor $\tilde{D}$ on $(X,M)$ is a \textit{prime divisor} if it is a prime divisor on $U$ or $\tilde{D}=\sum_{\sigma\in G/\mathrm{Stab}(\bfm,c)}\tilde{D}_{\sigma(\bfm,c)}$ for some $(\bfm,c)\in \Gamma_{M,\cC}$, where $\mathrm{Stab}(\bfm,c)\subset G$ is the stabilizer subgroup of $(\bfm,c)$.
		Finally, a $\mathbb{Q}$-divisor is called \textit{effective} if it is a nonnegative $\mathbb{Q}$-linear combination of prime divisors on $(X,M)$.
	\end{definition}
	\begin{notation}
		If $D,D'$ are two $\mathbb{Q}$-divisors on a pair $(X,M)$, then we write $D\geq D'$ if $D-D'$ is effective.
	\end{notation}
	In order to define the Picard group of a pair, we first need to define when two divisors on a pair are linearly equivalent. We will define this notion by introducing the pullback $\pr^*_M\colon \Div(X)\rightarrow \Div(X,M)$ from divisors on $X$ to divisors on $(X,M)$. In order to define this homomorphism, we first need to determine for every $(\bfm,c)\in \Gamma_{M,\cC}$ and every divisor $D\in \Div(X)$ what the coefficient $\mu((\bfm,c),D)$ of $\tilde{D}_{\bfm,c}$ in the pullback of $D$ should be.
	
	\begin{definition} \label{def: valuation (m,c)}
		Let $(X,M)$ be a smooth split pair over a field $K$ and let $(\bfm,c)\in \mathbb{N}^n_{\cC}$ such that $\gcd(\bfm)=1$. Let $\cO_{X,c}$ be the local ring at the closed subscheme $c\subset X$. Let $f_1,\dots,f_n\in \cO_{X,c}$ be local equations of the divisors $D_1,\dots,D_n$. We define $$v_{(\bfm,c)}\colon K(X)^\times\rightarrow \mathbb{Z}$$ to be the discrete valuation on $K(X)$ given by $v_{(\bfm,c)}(f_i)=m_i$ for all $i\in \{1,\dots, n\}$ and $v_{(\bfm,c)}(g)=0$ for all $g\in \cO_{X,c}^\times$.
		
		For a smooth pair $(X,M)$ over $K$ with splitting field $E$ and $(\bfm,c)\in \mathbb{N}^n_{\cC}$, we define $v_{(\bfm,c)}$ by restricting the corresponding discrete valuation $E(X)^\times\rightarrow \mathbb{Z}$ to $K(X)^\times$.
	\end{definition}
	Because $\sum_{i=1}^n D_i$ is a strict normal crossing divisor, it follows that $$\{f_i\mid i\in \{1,\dots,n\}, m_i\neq 0\}$$ is a minimal set of generators for the maximal ideal of the local ring $\cO_{X,c}$. Furthermore, if $i\in \{1,\dots,n\}$ satisfies $m_i=0$, then $f_i\in \cO_{X,c}^\times$. These facts together imply that $v_{(\bfm,c)}$ is indeed a well defined discrete valuation.
	\begin{remark} \label{remark: valuation Galois invariant}
		Note that on a smooth pair $(X,M)$ over $K$, the valuation $v_{(\bfm,c)}$ in Definition \ref{def: valuation (m,c)} only depends on the Galois orbit of $(\bfm,c)$, i.e. we have
		$$v_{(\bfm,c)}=v_{\sigma(\bfm,c)},$$
		for all $\sigma\in G$, where $G=\gal(E/K)$, where $E$ is the splitting field of $K$. This is because, the corresponding valuations on $E(X)^\times$ satisfy
		$$v_{(\bfm,c)}(f)=v_{\sigma(\bfm,c)}(\sigma(f)),$$
		for all $f\in E(X)^\times$.
	\end{remark}
	We will use the valuation $v_{(\bfm,c)}$ to define the pullback $\Div(X)\rightarrow \Div(X,M)$.
	\begin{definition} \label{def: mu}
		Let $D$ be a divisor on $X$ and let $(\bfm,c)\in \mathbb{N}^n_{\cC}\setminus\{(\mathbf{0},X)\}$. We define $$\mu((\bfm,c),E-E'):=\gcd(\bfm)\cdot v_{(\bfm/\gcd(\bfm),c)}(f),$$
		where $f\in K(X)^\times$ is a local equation for $D$ on some open set intersecting $c$.
	\end{definition}
	In particular, we have
	\begin{equation} \label{eq: concave 2}
		\mu((d\bfm,c), D)=d \mu((\bfm,c), D)
	\end{equation} for every positive integer $d$.
	\begin{definition}
		Let $(X,M)$ be a smooth pair over a field $K$.
		The \textit{pullback} of divisors on $X$ to divisors on $(X,M)$ is the group homomorphism $$\pr_M^*\colon \Div(X)\rightarrow \Div(X,M)$$ defined by
		$$D_i\mapsto \sum_{(\bfm,c)\in \Gamma_{M,\cC}} m_i\tilde{D}_{\bfm,c},$$
		for $i=1,\dots, n$ and
		$$D\mapsto D+\sum_{(\bfm,c)\in \Gamma_{M,\cC}} \mu((\bfm,c),D)\tilde{D}_{\bfm,c}$$
		for every divisor $D\in \Div(U)$.
	\end{definition}
	Note that by Remark \ref{remark: valuation Galois invariant} the image is indeed Galois-invariant, so the homomorphism is well defined.

	\begin{remark}
		Let $(\bfm,c)\in \mathbb{N}^n_{\cC}$. If $c$ is not contained in the support of a divisor $D$ on $X$, then
		$$\mu((\bfm,c),D)=0.$$ 
	\end{remark}
	
	\begin{remark}
		The pullback of divisors to $(X,M)$ is intimately related to the theory of weighted (stacky) blow-ups as described in \cite{QuRy22}. For $(\bfm,c)\in \mathbb{N}_{\cC}^n$, consider an open subset $X'\subset X$ such that the divisors $D_1,\dots,D_n$ are principal on $X'$ and $\tilde{D}_{\bfm}= \tilde{D}_{\bfm,c}$. Let $I_{\bullet}$ be the graded $\cO_{X'}$-subalgebra of $\cO_{X'}[t]$ generated by $f_1 t^{m_1},\dots, f_n t^{m_n}$, where $f_i=0$ is a local equation for $D_i$. If $\pi\colon \mathrm{Bl}_{I_{\bullet}}X'\rightarrow X'$ is the corresponding weighted blow-up as defined in \cite[Definition 3.2.1]{QuRy22} and $D\in \Div(X')$, then the coefficient of the exceptional divisor in $\pi^* D$ is equal to $\mu((\bfm,c), D)$.
	\end{remark}
	Note that $\Div(U)$ embeds into $\Div(X,M)$ in two distinct ways. The first is by the embedding $\Div(U)\hookrightarrow \Div(X,M)\colon D\mapsto D$ directly given by the definition of $\Div(X,M)$. The second one is given by composing $\Div(U)\hookrightarrow \Div(X)$ with the pullback $\Div(X)\rightarrow \Div(X,M)$ to obtain the map $D\mapsto \pr^*_M D$.
	\begin{definition}
		For a divisor $D\in \Div(U)$, we call $D\in \Div(X,M)$ the \textit{strict transform} of $D$ in $(X,M)$.
	\end{definition}
	This terminology is motivated by the fact that $\pr^*_M D-D$ is a divisor whose restriction to $U$ is trivial. This is illustrated by the following example.
	
	\begin{example}
		Let $(X,M)$ be a split pair obtained by taking $X$ to be a smooth variety, the divisors $D_1,D_2$ to be two smooth divisors on $X$ intersecting transversally with connected intersection $D_1\cap D_2$, and let $\fM\subset \mathbb{N}^2$ be the monoid generated by $(1,1)$. Let $\tilde{X}=\Bl_{D_1\cap D_2} X\setminus(\tilde{D}_1\cup \tilde{D}_2)$, where $\tilde{D}_1,\tilde{D}_2$ are the strict transforms of $D_1$ and $D_2$ respectively. Then the homomorphism $\Div(X,M)\rightarrow \Div(\tilde{X})$, given by sending $\Div(U)$ to itself and $\tilde{D}_{(1,1)}$ to the exceptional divisor, is an isomorphism respecting the pullbacks of divisors in $U$. Thus under this isomorphism, the strict transform of a divisor in $U$ in $\Div(\tilde{X})$ corresponds to the strict transform in $\Div(X,M)$.
	\end{example}
	
	While the function $\mu\colon \mathbb{N}^n_{\cC}\times \Div(X)\rightarrow \mathbb{Z}$ is linear in the divisor argument, it is not linear in the first argument as the next examples show.
	\begin{example}
		Let $X=\mathbb{P}^2$, $n=2$, $D_1=\{X_1=0\}$, $D_2=\{X_2=0\}$, $\fM=\mathbb{N}^2\setminus\{(1,0),(0,1)\}$. Then $D_1\cap D_2=\{(0:0:1)\}$, and the divisor $D=\{X_1=X_2\}$ has pullback $\pr^*_M D=D+\tilde{D}_{(1,1)}$. In particular, $$1=\mu((1,1),D)\geq \mu((1,0),D)+\mu((0,1),D)=0$$
		For an example with no multiplicities equal to zero, we can take $$3=\mu((3,3),D)\geq \mu((2,1),D)+\mu((1,2),D)=2.$$
	\end{example}
	However, $\mu$ is a \textit{concave} function in the second argument, as the next lemma shows. We will exploit this fact in Lemma \ref{lemma: reduction to small generators} to understand the Fujita invariant and the $b$-invariant of line bundles. 
	\begin{lemma} \label{lemma: concave}
		Let $(X,M)$ be a smooth pair. Then for all $(\bfm,c),(\bfm',c')\in \mathbb{N}^n_{\cC}$ and for every effective divisor $D$ on $X$,
		\begin{equation} \label{eq: concave 1}
			\mu((\bfm+\bfm',\tilde{c}),D)\geq\mu((\bfm,c),D)+\mu((\bfm',c'),D),
		\end{equation}
		for every connected component $\tilde{c}\in \cC_{\bfm+\bfm'}$ of $c\cap c'$.
		Therefore, for all $\lambda,\lambda'\in \mathbb{Q}$ with $\lambda \bfm\in \mathbb{N}^n$ and $\lambda\bfm'\in \mathbb{N}^n$,
		\begin{equation} \label{eq: concave 3}
			\mu((\lambda\bfm+\lambda'\bfm',\tilde{c}),D)\geq \lambda\mu((\bfm,c),D)+\lambda'\mu((\bfm',c'),D).
		\end{equation}
	\end{lemma}
	\begin{proof} 
		By the equality \eqref{eq: concave 2}, inequality \eqref{eq: concave 3} follows directly from the inequality \eqref{eq: concave 1}.
		For any tuple $(\overline{\bfm},\overline{c})\in \fM_{\cC}$ with $\tilde{c}\subset \overline{c}$ and a nonnegative integer $\mu$, let $I_{(\overline{\bfm},\overline{c})}^\mu$ be the ideal of $\cO_{X,\tilde{c}}$ consisting of the $g\in \cO_{X,\tilde{c}}$ with $$\gcd(\overline{\bfm})v_{(\overline{\bfm}/\gcd(\overline{\bfm}),\overline{c})}(g)\geq \mu.$$ In particular, for a local equation $f\in \cO_{X,\tilde{c}}$ for $D$, we have $f\in I_{(\overline{\bfm},\overline{c})}^\mu$ if and only if $\mu((\overline{\bfm},\overline{c}),D)\geq \mu$.
		Therefore, to prove the statement, it suffices to prove that $$I_{(\bfm,c)}^\mu\cap I_{(\bfm,c)}^{\mu'}\subset I_{(\bfm,\tilde{c})}^{\mu+\mu'}$$
		for all positive integers $\mu,\mu'$.
		The ideal $I_{(\overline{\bfm},\overline{c})}^\mu$ is generated by elements $\prod_{i=1}^n f_i^{a_i}$ such that $\sum_{i=1}^n a_i\overline{m}_i\geq \mu$.
		
		Since $\{f_i\mid i\in \{1,\dots,n\},\, m_i>0\}$ is a minimal set of generators for the maximal ideal of $\cO_{X,\tilde{c}}$, it follows that the intersection $I_{(\bfm,c)}^\mu\cap I_{(\bfm,c)}^{\mu'}$ is generated by elements $x$ which can be written as $x=\prod_{i=1}^n f_i^{a_i}$ such that both $\sum_{i=1}^n a_im_i\geq \mu$ and $\sum_{i=1}^n a_im'_i\geq \mu'$. But this implies $\sum_{i=1}^n a_i(m_i+m_i')\geq \mu+\mu'$ and thus $x\in I_{(\bfm,\tilde{c})}^{\mu+\mu'}$, so we have shown $I_{(\bfm,c)}^\mu\cap I_{(\bfm,c)}^{\mu'}\subset I_{(\bfm,\tilde{c})}^{\mu+\mu'}$ as desired.
	\end{proof}
	
	
	\subsection{The Picard group of a pair} \label{section: Picard group pair}
	In this section we introduce the Picard group and the effective cone of a pair.
	\begin{definition}
		We say that a divisor on $(X,M)$ is \textit{principal} if it is the image of a principal divisor on $X$ under the homomorphism $\pr^*_M\colon \Div(X)\rightarrow \Div(X,M)$. We say that two divisors $D,D'$ on $(X,M)$ are \textit{linearly equivalent} if $D-D'$ is a principal divisor.
	\end{definition}
	
	\begin{definition}
		We define the \textit{Picard group} of $(X,M)$ as
		$$\Pic(X,M)=\Div(X,M)/\{\text{principal divisors}\}.$$  
	\end{definition}
	By the definition, the pullback $\pr_M^*\colon \Div(X)\rightarrow \Div(X,M)$ induces a homomorphism
	$$\Pic(X)\rightarrow \Pic(X,M),$$
	which we will also often denote by $\pr^*_M$, by abuse of notation.
	We will denote the induced homomorphisms on $\mathbb{Q}$-divisors $\Div(X)_{\mathbb{Q}}\rightarrow \Div(X,M)_{\mathbb{Q}}$ and on $\mathbb{Q}$-divisor classes $\Pic(X)_{\mathbb{Q}}\rightarrow \Pic(X,M)_{\mathbb{Q}}$ by $\pr^*_M$ as well. For a ($\mathbb{Q}$-)divisor $D$, we will denote the corresponding ($\mathbb{Q}$-)divisor class by $[D]$.
	\begin{definition}
		We define two $\mathbb{Q}$-divisors $D,D'$ on a pair $(X,M)$ to be \textit{$\mathbb{Q}$-linearly equivalent} if the image of $D-D'$ in $\Pic(X,M)_{\mathbb{Q}}$ is $0$.
	\end{definition}
	\begin{example} \label{example: Picard group root stack}
		If $(X,M)$ is a smooth split pair corresponding to Darmon points with associated root stack $(X,\sqrt[\bfm]{D})$ as in \cite[\S 8]{Moe24}, then \cite[Corollary 3.12]{Cad07} implies that the map $\Pic(X,M)\rightarrow \Pic(X,\sqrt[\bfm]{D})$, given by sending $\tilde{D}_{m_i\bfe}$ to the Cartier divisor $\frac{1}{m_i} D_i$ for all $i\in \{1,\dots,n\}$ with $m_i<\infty$, is an isomorphism. Furthermore, this isomorphism is compatible with the pullback homomorphisms of $\Pic(X)\rightarrow \Pic(X,M)$ and $\Pic(X)\rightarrow  \Pic(X,\sqrt[\bfm]{D})$.
	\end{example}
	By taking all multiplicities $m_1,\dots,m_n$ to be infinite, we obtain the following special case of the previous example.
	\begin{example} \label{example: Picard group open}
		Let $X$ be a smooth variety and let $U\subset X$ be an open subvariety such that $X\setminus U$ is a strict normal crossings divisor. If $(X,M)$ is the smooth pair corresponding to integral points on $U$, then $\Div(X,M)=\Div(U)$ so $\Pic(X,M)=\Pic(U)$.
	\end{example}
	By construction of the Picard group of a pair, we have a natural surjection $\Pic(X,M)\rightarrow \Pic(U)$ by restricting a divisor on $(X,M)$ to $U$. More generally, we can define restriction maps between the Picard groups of pairs.
	\begin{definition} \label{def: restriction divisors}
		Let $(X,M)$ and $(X,M')$ be two smooth pairs for the same choice of divisors $D_1,\dots,D_n$, such that $\Gamma_{M',\cC}\subset \Gamma_{M,\cC}$. Then we define the \textit{restriction (of divisors) to $(X,M')$} to be the homomorphism
		$$\Div(X,M)\rightarrow \Div(X,M')$$
		which sends $\Div(X,M')\subset \Div(X,M)$ to itself and sends $D_{(\bfm,c)}$ to $0$ if $(\bfm,c)\in \Gamma_{M,\cC}$ but $\bfm\not\in \Gamma_{M'}$.
		This homomorphism induces a homomorphism
		$$\Pic(X,M)\rightarrow \Pic(X,M'),$$
		which we will also refer to as the \textit{restriction (of divisor classes) to $(X,M')$}.
	\end{definition}
	Note that these restrictions are always surjective homomorphisms.
	The cokernel of the homomorphism $\pr^*_M\colon \Pic(X)\rightarrow \Pic(X,M)$ is generated by $[\tilde{D}_{\bfm,c}]$ for $(\bfm,c)\in\Gamma_{M,\cC}$, and is thus finitely generated. Consequently, if $\Pic(X)$ is finitely generated, then $\Pic(X,M)$ will be as well, as it requires at most $\#\Gamma_{M,\cC}$ additional generators.
	\begin{example} \label{example: Picard group Darmon points}
		If $(X,M)$ is a smooth pair for the Darmon points on the Campana pair $\left(X,\sum_{i=1}^n \left(1-\frac{1}{m_i} \right)D_i\right)$, then $\Div(X,M)$ is naturally identified with $\Div(U)\oplus \mathbb{Z}\left(\frac{1}{m_1} D_1\right)\oplus \dots \oplus \mathbb{Z}\left(\frac{1}{m_n}D_n\right)$, and the homomorphism $\Pic(X)\rightarrow \Pic(X,M)$ is injective with cokernel
		$$\mathbb{Z}/m_1\mathbb{Z}\times\dots\times \mathbb{Z}/m_n\mathbb{Z}.$$
	\end{example}
	\begin{example}
		In the previous example, if we take $X=\mathbb{P}^{n-1}_K$, and we let $D_1,\dots,D_n$ be the coordinate hyperplanes, then $$\Pic(X,M)\cong \mathbb{Z}^n/\left\{(a_1,\dots,a_n)\in \mathbb{Z}^n\,\middle\vert \, \sum_{i=1}^n \frac{a_i}{m_i}=0\right\}.$$
		In particular, if $\gcd(m_i,m_j)=1$ for all distinct $i,j\in \{1,\dots,n\}$, then $\Pic(X,M)\cong \mathbb{Z}$. On the other hand, if $m=m_1=\dots=m_n$, then $\Pic(X,M)\cong \mathbb{Z}\times (\mathbb{Z}/m\mathbb{Z})^{n-1}$. This shows that $\Pic(X,M)$ can contain nontrivial torsion even when $\Pic(X)$ is torsion-free.
	\end{example}
	The previous example showed that the torsion of $\Pic(X,M)$ is slightly subtle, and depends on $M$, rather than only on $X$. However, for a proper pair, the rank of its Picard group is very simple to describe.
	\begin{proposition} \label{prop: rank pic for proper pair}
		If $(X,M)$ is a proper smooth pair over a field $K$, then $\pr^*_M$ is injective on $\Div(X)$ and on $\Pic(X)$. Moreover,
		$$\rank \Pic(X,M)=\rank \Pic(X)+\#\left(\Gamma_{M,\cC} \setminus \{(d_1\mathbf{e}_1,D_1),\dots, (d_n\mathbf{e}_n,D_n)\}\right)/G,$$
		where $d_1,\dots,d_n$ are integers as in Definition \ref{def: proper pair}.
		If $(X,M)$ is split then this simplifies to
		$$\rank \Pic(X,M)= \#\Gamma_{M,\cC}-n.$$
	\end{proposition}
	\begin{proof}
		Since $(X,M)$ is proper, for each $i=1,\dots, n$ there exists an element $(d_i\mathbf{e}_i,D_i)\in \Gamma_{M,\cC}$ for some integer $d_i$. In particular, the coefficient of $\tilde{D}_{d_i\mathbf{e}_i,D_i}$ in $\pr_M^*D_i$ is $d_i$, while for all $j\in \{1,\dots,n\}$ different from $i$, the coefficient of $\tilde{D}_{d_i\mathbf{e}_i,D_i}$ in $\pr_M^*D_j$ is zero, as $D_i$ is not contained in $D_j$. Similarly, for all $D\in \Div(U)$, the the coefficient of $\tilde{D}_{d_i\mathbf{e}_i,D_i}$ in $\pr_M^*D$ is also zero. Since, furthermore, the restriction of $\pr_M^*$ to $\Div(U)$ is injective, $\pr^*_M$ is indeed injective. This directly implies that the homomorphism on the Picard groups is also injective.
		
		If we let $(X,M')\subset (X,M)$ be the divisorial pair with $\fM'=\{(0,\dots,0),d_1\mathbf{e}_1,\dots,d_n\mathbf{e}_n\}$, then we can consider the restriction $\Pic(X,M)\rightarrow \Pic(X,M')$. The kernel of this restriction is the free abelian group consisting of Galois-invariant linear combinations of the classes $[\tilde{D}_{\bfm,c}]$ for $(\bfm,c)\in \Gamma_{M,\cC}$ with $\bfm\not\in \Gamma_{M'}$. In particular, it is a free abelian group of rank $\#\left(\Gamma_{M,\cC} \setminus \{(d_1\mathbf{e}_1,D_1),\dots, (d_n\mathbf{e}_n,D_n)\}\right)/G$. By Example \ref{example: Picard group Darmon points}, the cokernel of $\pr^*_{M'}\colon \Pic(X)\rightarrow \Pic(X,M')$ is finite. As $(X,M')$ is proper, this implies $\rank \Pic(X,M)=\rank \Pic(X)$, so the rank of $\Pic(X,M)$ is $\rank \Pic(X)+\#\left(\Gamma_{M,\cC} \setminus \{(d_1\mathbf{e}_1,D_1),\dots, (d_n\mathbf{e}_n,D_n)\}\right)/G$. If $(X,M)$ is split, then the action is trivial and we obtain the desired formula for the rank.
	\end{proof}
	The Picard group of a pair $(X,M)$ injects into the Picard group of the base change $(X_{\overline{K}}, M)$, up to torsion.
	\begin{proposition} \label{prop: Picard group injects into closure}
		Let $(X,M)$ be a smooth pair. Then the natural homomorphism
		$$\Pic(X,M)_{\mathbb{Q}}\rightarrow \Pic(X_{\overline{K}},M)_{\mathbb{Q}}$$
		is an injection.
	\end{proposition}
	\begin{proof}
		Suppose that $D\in \Div(X,M)$ is an element whose image in $\Pic(X_{\overline{K}},M)$ is trivial. This implies that $D$ is the pullback of $\div f$ for $f\in \overline{K}(X)^\times$. The rational function $f$ is defined over some Galois extension $E/K$, and $\tilde{f}=\prod_{\sigma\in \gal(E/K)} \sigma(f)\in K(X)^\times$. Since $D$ is invariant under the Galois action it follows that the pullback of $\div(\sigma(f))$ to $(X_E,M)$ is $D$ itself. Therefore, it follows that $\deg(E/K)D\in \Div(X,M)$ is a principal divisor, so $[D]\in \Pic(X,M)$ is torsion, as desired.
	\end{proof}
	\begin{remark}
		In general, the Picard group of $(X,M)$ does not need to inject into $(X_{\overline{K}},M)$, as torsion elements may be mapped to zero. For example, if $X=\mathbb{P}^1_{\mathbb{Q}}$ and $D=\{(1:i),(1:-i)\}$, then $\Pic(X\setminus D)\cong \mathbb{Z}/2\mathbb{Z}$, but $\Pic(X_{\overline{\mathbb{Q}}}\setminus D_{\overline{\mathbb{Q}}})\cong 0$.
	\end{remark}
	\subsection{The canonical divisor and the effective cone}
	Now we will define the canonical divisor class of a pair.
	\begin{definition} \label{def: Canonical divisor}
		For a smooth pair $(X,M)$ over a field $K$ of characteristic $0$, we define the \textit{ramification divisor} of $(X,M)$ to be the effective divisor
		$$R:=\sum_{(\bfm,c)\in \Gamma_{M,\cC}} \left(-1+\sum_{i=1}^n m_{i}\right)\tilde{D}_{\bfm,c}.$$
		The \textit{canonical (divisor) class} $K_{(X,M)}\in \Pic(X,M)$ of a smooth pair $(X,M)$ is
		\begin{align*}
			K_{(X,M)}:=&\pr^*_M K_X+R \\
			=&\pr^*_M \left(K_X+\sum_{i=1}^n [D_i]\right)-\sum_{(\bfm,c)\in \Gamma_{M,\cC}} [\tilde{D}_{\bfm,c}].
		\end{align*}
	\end{definition}
	Thus the canonical class is defined in such a way that it satisfies an analogue of the Hurwitz formula for morphisms of curves (see e.g. \cite[Chapter IV, Proposition 2.3]{Har77}). If $(X,M)$ is a smooth pair corresponding to Darmon points on a Campana pair $(X,D_{\bfm})$, then the canonical divisor of $(X,M)$ agrees with the canonical divisor of the root stack $(X,\sqrt[\bfm]{D})$, see for example \cite[Proposition 5.5.6]{VoZB22} for the case when $X$ is a curve.

	\begin{assumption}
		For the rest of this article we assume that $X$ is a rationally connected proper variety.
	\end{assumption}
	There are two reasons for this assumption. One reason is that Conjecture \ref{conj: generalized Manin conjecture} only applies to rationally connected varieties. Furthermore, if $X$ is rationally connected then the Albanese variety of $X$ is trivial, and thus its dual $\Pic^0(X)$ is also trivial. Hence, $\Pic(X,M)$ is a finitely generated abelian group for every smooth pair $(X,M)$.
	\begin{definition}
		The \textit{effective cone of a smooth pair $(X,M)$} is the cone $$\Eff^1(X,M)\subset \Pic(X,M)_{\mathbb{R}}$$ generated by effective divisors on $(X,M)$. Its topological closure is the \textit{pseudo-effective cone of $(X,M)$} $$\overline{\Eff}^1(X,M)\subset \Pic(X,M)_{\mathbb{R}}.$$
		We also similarly write $\Eff^1(X)$ and $\overline{\Eff}^1(X)$ for the effective and pseudo-effective cones of $X$.
	\end{definition}
	
	The effective cone of a proper pair is strictly convex, as the next proposition shows.
	\begin{proposition} \label{prop: effective cone strictly convex}
		Let $(X,M)$ be a smooth proper pair. Then $\Eff^1(X,M)$ is strictly convex, i.e. $$\Eff^1(X,M)\cap -\Eff^1(X,M)=\{\mathbf{0}\}.$$
	\end{proposition}
	\begin{proof}
		We assume that $K$ is algebraically closed so $(X,M)$ is split, which we can do without loss of generality by Proposition \ref{prop: Picard group injects into closure}.
		We will argue by contradiction. Suppose that there exists an element $E\in \Eff^1(X,M)$ satisfying $-E\in \Eff^1(X,M)$. Since $\Eff^1(X,M)$ is generated by effective divisors, $\Eff^1(X,M)\cap \mathbb{Q}$ is dense in $\Eff^1(X,M)$. This implies that there exists an integer $m$ and nonzero effective divisors $D_1,D_2$ on $(X,M)$ such that $mE-[D_1], -mE-[D_2]\in \Eff^1(X,M)$. The sum $D:=D_1+D_2$ is a nonzero effective divisor such that $-[D]\in \Eff^1(X,M)$. This implies that $-D$ is $\mathbb{Q}$-linearly equivalent to an effective $\mathbb{Q}$-divisor $D'$. It follows that $D+D'$ is an effective $\mathbb{Q}$-divisor which is $\mathbb{Q}$-linearly equivalent to $0$. Thus there exists a positive integer $m$ such that $m(D+D')$ is linearly equivalent to an effective divisor and $m(D+D')=\pr_M^* \div(f)$ for some rational function on $X$. The function $f$ cannot have any poles on $U$, and since $(X,M)$ is proper, there exist positive integers $m_1,\dots, m_n$ such that $m_1\bfe_1,\dots, m_n\bfe_n\in \Gamma_M$. This implies that $f$ cannot have any poles at the divisors $D_1,\dots,D_n$ either, which implies that $f$ is a regular function on $X$ and thus $f$ is constant, as $X$ is proper. We conclude that $D+D'=0$. This contradicts the fact that $D$ is nonzero, so $E$ cannot exist.
	\end{proof}
	
	Using the effective cone, we will now define the Fujita invariant and the $b$-invariant of a pair.
	\begin{definition} \label{def: Fujita invariant}
		Let $(X,M)$ be a smooth pair over a field $K$ of characteristic $0$. Let $L$ be a nef and big $\mathbb{Q}$-divisor class on $X$. We define the \textit{Fujita invariant} of $(X,M)$ with respect to $L$ to be
		$$a((X,M),L)=\mathrm{inf}\{t\in \mathbb{R}\mid t\pr_M^* L+K_{(X,M)}\in \overline{\Eff}^1(X,M)\}.$$
		We call $a((X,M),L)\pr_M^* L+K_{(X,M)}\in \overline{\Eff}^1(X,M)$ the \textit{adjoint divisor class of $L$ with respect to $(X,M)$}.
		We define the \textit{$b$-invariant} $b(K,(X,M),L)$ to be the codimension of the minimal supported face of $\overline{\Eff}^1(X,M)$ which contains the adjoint divisor class of $L$ with respect to $(X,M)$.
	\end{definition}
	Note that the Fujita invariant is strictly positive if and only if $K_{(X,M)}$ is not pseudoeffective.
	
	There need not exist a nef and big $\mathbb{Q}$-divisor class $L$ such that $\pr_M^*L=-K_{(X,M)}$, so the $b$-invariant may be strictly smaller than the rank of the Picard group of $(X,M)$ for all choices of $L$, as the next example illustrates.
	\begin{example}
		Let $(X,M)$ be the smooth split pair over $K$ corresponding to the Campana points on $\left(X,\sum_{i=1}^n\left(1-\frac{1}{m_i}\right)D_i\right)$ for a rationally connected variety $X$. Then Proposition \ref{prop: rank pic for proper pair} implies that
		$$\rank \Pic(X,M)=\rank\Pic(X)+\sum_{i=1}^n (m_i-1),$$
		where we use the description of $\Gamma_M$ from Example \ref{example: generators Darmon and Campana} and the fact that the set of components $\cC_{m\bfe_i}$ is just $\{D_i\}$ for any positive integer $m$ and $i\in \{1,\dots, n\}$.
		However, for any divisor $L\in \Div(X)$, the coefficient of $\tilde{D}_{m\bfe_i}$ in $L$ for $m_i\leq m\leq 2m_i-1$ is simply $a_im$, where $a_i$ is the coefficient of $D_i$ in $L$. This follows from the equality \eqref{eq: concave 2} combined with the fact that $\mu(m \bfe_i, D)=0$ for all prime divisors $D\neq D_i$ on $X$. Therefore, if the coefficient of $\tilde{D}_{m_i\bfe_i}$ in $\pr_M^* L+R=\pr_M^* L+\sum_{i=1}^n \sum_{m=m_i}^{2m_i-1} (m-1) \tilde{D}_{m\bfe_i}$ is nonnegative for some $i\in \{1,\dots, n\}$, then the coefficient of $\tilde{D}_{m\bfe_i}$ in $\pr_M^* L+R$ is positive for integers $m$ with $m_i<m\leq 2m_i-1$.
		Consequently, we have $b(K,(X,M),L)\leq \rank \Pic(X)$ for all big and nef divisors $L$.
	\end{example}
	In fact, for computing the Fujita invariant and the $b$-invariant, we only need to consider small generators $(\bfm,c)\in \Gamma_{M,\cC}$, which often considerably simplifies computations. In particular, we only need to consider the minimal elements for the partial ordering on $\Gamma_{M,\cC}$ given by $(\bfm,c)\leq(\bfm',c')$ if $\bfm\leq \bfm'$ and $c\supset c'$. The following lemma shows that we need to consider even less elements in $\Gamma_{M,\cC}$.  For a stratum $c$ we set $$\Gamma_{M,c}=\{(\bfm,\tilde{c})\in \Gamma_{M,\cC}\mid c\subset \tilde{c} \}.$$
	\begin{lemma} \label{lemma: reduction to small generators}
		Let $(X,M)$ be a smooth pair over a field $K$ and let $L$ be a big and nef $\mathbb{Q}$-divisor on $X$. For a stratum $c$, let $P_c$ be the polyhedron given as the convex hull of the set $$\{\bfm+\bfx\in\mathbb{R}^n\mid (\bfm, \tilde{c})\in\Gamma_{M,c}, \bfx\in \mathbb{R}_{\geq 0}^n\}$$ and let $\partial P_c$ and $V(P_c)$ be its boundary and its set of vertices, respectively. Define pairs $(X,M'')\subset (X,M')\subset (X,M)$ by setting
		$$\fM'_{\cC}=\{(\bfm,\tilde{c}) \in \Gamma_{M,\cC}\mid \text{there is $c$ such that } c\subset \tilde{c}, \bfm\in \partial P_c\}$$ and 
		$$\fM''_{\cC}=\{(\bfm,\tilde{c}) \in \Gamma_{M,\cC}\mid \text{there is $c$ such that } c\subset \tilde{c}, \bfm\in V(P_c)\}.$$
		Then
		$$a((X,M),L)=a((X,M'),L)=a((X,M''),L)$$
		and
		$$b(K,(X,M),L)=b(K,(X,M'),L).$$
	\end{lemma}
	\begin{remark} \label{remark: reduction to small generators divisorial}
		Note that here it suffices to consider the smallest strata for the polyhedra, since if $c\subset c'$, then $P_{c'}\subset P_c$. Furthermore, for a divisorial pair $(X,M)$, the pairs $(X,M')$ and $(X,M'')$ are both divisorial, with sets of multiplicities $\fM'=\Gamma_M \cap \partial P$ and $\fM''=\Gamma_M \cap V(P)$, where $P$ is the polyhedron given as the convex hull of the set $$\{\bfm+\bfx\in\mathbb{R}^n\mid \bfm \in\Gamma_{M}, \bfx\in \mathbb{R}_{\geq 0}^n\}.$$
	\end{remark}
	\begin{proof}
		If $\Gamma_{M,\cC}=\emptyset$, then $(X,M)=(X,M')=(X,M'')$ is the divisorial pair with $\fM=\{(0,\dots,0)\}$, so the lemma is trivially satisfied. Now assume $\Gamma_{M,\cC}\neq \emptyset$ and $c$ is a stratum, and let $(\bfm,\tilde{c})\in \Gamma_{M,c}$ be an element such that $\bfm$ is not on the boundary of the polyhedron $P_c$. Then there exist $(\bfm_1,c_1),\dots,(\bfm_T,c_T)\in \Gamma_{M',c}$ such that $\bfm=\sum_{t=1}^T \lambda_t\bfm_t$ for some real numbers $\lambda_1,\dots,\lambda_T>0$ satisfying $\sum_{t=1}^T\lambda_t>1$. If $L'\in \Div(X)_{\mathbb{Q}}$ is a $\mathbb{Q}$-divisor, such that for all $t\in \{1,\dots,T\}$ the coefficient of $\tilde{D}_{\bfm_t,c_t}$ in $\pr^*_M L'\in \Div(X,M)$ is at least $1$, then Lemma \ref{lemma: concave} implies that the coefficient of $\tilde{D}_{\bfm,c}$ in $\pr^*_M L'\in \Div(X,M)$ is at least $\sum_{t=1}^T\lambda_t>1$.
		This implies that for any $L\in \Pic(X)_{\mathbb{Q}}$ and $a\in \mathbb{Q}$ such that $$a\pr^*_M L+K_{(X,M)}=\pr^*_M\left(L+K_X+\sum_{i=1}^n [D_i]\right)-\sum_{(\bfm,c)\in \Gamma_{M,\cC}}[\tilde{D}_{\bfm,c}]\in \Pic(X,M)_{\mathbb{Q}}$$
		is an effective $\mathbb{Q}$-divisor class, $a\pr^*_M L+K_{(X,M)}$ is represented by an effective $\mathbb{Q}$-divisor such that the coefficient of $\tilde{D}_{\bfm,c}$ is at least $-1+\sum_{t=1}^T\lambda_t>0$.
		This implies that the minimal face of $\overline{\Eff}^1(X,M)$ containing the adjoint divisor class of an big and nef $\mathbb{Q}$-divisor class $L$ contains $\tilde{D}_{\bfm,c}$. Thus we see that $a((X,M),L)=a((X,M'),L)$ and $b(K,(X,M),L)=b(K,(X,M'),L)$, as desired.
		The proof of the equality $a((X,M),L)=a((X,M''),L)$ is entirely analogous.
	\end{proof}
	More generally, the Fujita invariant and the $b$-invariant are smaller on smaller pairs.
	\begin{proposition} \label{prop: Fujita invariant inclusion}
		Let $(X,M)\subset(X,M')$ be smooth pairs over a field $K$ of characteristic $0$ and let $L$ be a big and nef $\mathbb{Q}$-divisor on $X$. Then we have
		$$(a((X,M),L), b(K,(X,M),L))\leq (a((X,M'),L), b(K,(X,M'),L))$$
		in the lexicographic ordering.
	\end{proposition}
	\begin{proof}
		This is essentially the same argument as in Lemma \ref{lemma: reduction to small generators} by using Lemma \ref{lemma: concave}.
	\end{proof}
	
	\begin{remark}
		By Lemma \ref{lemma: reduction to small generators}, the Fujita invariant and the $b$-invariant only depend on the polyhedra generated by $\fM_{\cC}$, rather than $\fM_{\cC}$ itself. Therefore the Conjecture \ref{conj: generalized Manin conjecture} satisfies a form of purity: the order of growth of the counting function only depends on the smallest elements in $\fM_{\cC}$.
	\end{remark}
	\begin{example}
		As a simple example for the previous remark, Conjecture \ref{conj: generalized Manin conjecture} implies that the $\mathbb{Q}$-rational points in projective space with coordinates both squarefree and pairwise coprime have a positive density in the full set of rational points, when the points are ordered by their Weil height.
	\end{example}
	There is a natural generalization of rigid divisors to pairs.
	\begin{definition} \label{def: adjoint rigid}
		Let $D\in \Div(X,M)_\mathbb{Q}$ be a $\mathbb{Q}$-divisor on a smooth pair $(X,M)$ with $X$ rationally connected. We say that $D$ is \textit{rigid} if $D$ is effective and it is the only effective $\mathbb{Q}$-divisor in its $\mathbb{Q}$-linear equivalence class. For a big and nef $\mathbb{Q}$-divisor $L$ on $X$, we say that $L$ is \textit{adjoint rigid} with respect to $(X,M)$ if $a((X,M),L)>0$ and the adjoint divisor class $a((X,M),L)\pr_M^* L+K_{(X,M)}$ is represented by a rigid $\mathbb{Q}$-divisor.
	\end{definition}
	

	\section{Rationally connected pairs} 
	\label{section: Manin's conjecture for M-points}
	In modern times, Manin's Conjecture is often formulated for rationally connected varieties, see for example \cite[Conjecture 1.2]{LeSeKuSh22}. The notion of rationally connected varieties (see e.g. \cite[Definition-Remark 2.2]{KMM92}) has a natural extension to smooth proper pairs.
	
	\begin{definition} \label{def: rationally connected pair}
		A smooth proper pair $(X,M)$ over a field $K$ is \textit{rationally connected} if there exists a nonempty open subvariety $V\subset X$ such that for each algebraically closed field $L/K$ and every two points $p_1, p_2\in V(L)$, there exists a rational curve $C\subset X_L$ containing both points such that $C$ is the image of a morphism $f\in(X,M_{\mon})(\mathbb{P}^1_L)$, where $(X,M_{\mon})$ is the pair from Notation \ref{notation: M_mon}.
	\end{definition}
	
	In other words, a pair $(X,M)$ is rationally connected if for any two general points on $X$ there is a projective rational curve passing through them respecting the conditions imposed by $M_{\mon}$. In particular, if $(X,M)$ is rationally connected, then $X$ is rationally connected as well. The projectivity of the curve is crucial here, as any curve $C$ on $X$ not contained in the divisors $D_1,\dots,D_n$ has a nonempty open subset $C'\subset C$ avoiding these divisors.
	
	One reason to consider such pairs in Conjecture \ref{conj: generalized Manin conjecture} is that they have a good reason for having plenty of $\cM$-points after an extension of the ground field. This is because the images of rational points under $f\in(X,M_{\mon})(\mathbb{P}^1_K)$ are $\cM$-points over $\cO_S$ for some finite set of places $S\subset \Omega_K$, as the next proposition shows.
	\begin{proposition}
		Let $(X,M)$ be a pair over a number field $K$ and let $f\in(X,M_{\mon})(\mathbb{P}^1_K)$. Then for every finite set of places $S\subset \Omega_K$ and every $\cO_S$-integral model $(\cX,\cM)$, there exists a finite set of places $S'\supset S$ such that $f$ extends to $\tilde{f}\in (\cX,\cM_{\mon})(\mathbb{P}^1_{\cO_{S'}})$.
		Furthermore, for any $P\in \mathbb{P}^1_K(K)$, $$f(P)\in(\cX,\cM_{\mon})(\cO_{S'})$$
		if $f(P)$ lies in $U$.
	\end{proposition}
	\begin{proof}
		By spreading out \cite[Theorem 3.2.1]{Poo17}, we can find a set of places $S'$ containing $S$ such that $f\colon \mathbb{P}^1_K\rightarrow X$ lifts to a morphism $\tilde{f}\colon \mathbb{P}^1_{\cO_{S'}}\rightarrow \cX$. The pullbacks $\tilde{f}^*\cD_1,\dots,\tilde{f}^*\cD_n$ are effective divisors on $\mathbb{P}^1_{\cO_{S'}}$. By further enlarging $S'$ if necessary, we can ensure that these divisors have no components supported above a prime $\fp$ in $\cO_{S'}$, which ensures that $\tilde{f}\in (\cX,\cM_{\mon})(\mathbb{P}^1_{\cO_{S'}})$. Let $P\in \mathbb{P}_K^1(K)$. By the valuative criterion of properness, $P$ corresponds to an unique integral point $\tilde{P}\colon \spec \cO_{S'}\rightarrow \mathbb{P}^1_{\cO_{S'}}$. Now it follows from the functoriality discussed in Remark \ref{remark: (X,M) as a functor} that $\tilde{f}\circ \tilde{P}\in (\cX,\cM_{\mon})(\cO_{S'})$.
	\end{proof}
	
	We finish the section by noticing that toric pairs are rationally connected, giving many examples of rationally connected pairs.
\begin{proposition}
	Let $(X,M)$ be a smooth proper pair over a field $K$ such that $X$ is a toric variety and $D_1,\dots,D_n$ are torus-invariant prime divisors. Then $(X,M)$ is rationally connected.
\end{proposition}
\begin{proof}
	
	
	Without loss of generality, we can assume that $\fM$ is a monoid and that $X$ is a split toric variety (so that $(X,M)$ is a split pair). Let $L$ be an algebraically closed field containing $K$ and let $L(t)$ be the field of rational functions over $L$ in the variable $t$. Write $(X_{L(t)},M_{L(t)})$ for the pair over $L(t)$ obtained by base changing $X$ and $D_1,\dots, D_n$ to $L(t)$. This pair has the obvious integral model $(X_{\mathbb{P}^1_L},M_{\mathbb{P}^1_L})$ over $\mathbb{P}^1_L$ given by base changing $(X,M)$ to $\mathbb{P}^1_L$. Now \cite[Theorem 1.3]{Moe24} implies that the pair $(X_{L(t)},M_{L(t)})$ satisfies $M_{L(t)}$-approximation off the place $(1:0)\in \mathbb{P}_L^1$. Consequently, the embedding $$(X_{\mathbb{P}^1_L},M_{\mathbb{P}^1_L})(\mathbb{A}_L^1)=(X,M)(\mathbb{A}_L^1)\rightarrow \prod_{v\in \Omega_{L(t)}\setminus \{(1:0)\}} (X,M)(\cO_{v})$$
	has dense image. In particular $(X,M)(\mathbb{A}_L^1)$ has dense image in $$(X,M)(L\lParen t \rParen)\times (X,M)(L\lParen t-1 \rParen).$$ Since $U(L)\subset (X,M)(L\lParen t \rParen)$ and $U(L)\subset (X,M)(L\lParen t-1 \rParen)$, this implies that any two points $p_1,p_2\in U(L)$ are contained in the image of some $f\in(X,M)(\mathbb{A}_L^1)$. As $\fM$ is a monoid and $(X,M)$ is proper, $\fM$ contains $a\mult_{(1:0)}(f)\in \mathbb{N}^n$ for some positive integer $a$. Let $g\colon \mathbb{P}^1_L\rightarrow \mathbb{P}^1_L$ be the map given by $(x:y)\mapsto (x^a:y^a)$. Then $f\circ g\in (X,M)(\mathbb{P}^1_L)$ and it contains the points $p_1$ and $p_2$ in its image. This implies that $(X,M)$ is rationally connected.
\end{proof}
	
	\section{Quasi-Campana points and the log-canonical divisor}
	\label{section: quasi-Campana pairs}
	In this section we specialize the Conjecture \ref{conj: generalized Manin conjecture} to Campana points, weak Campana points and Darmon points as defined in Definition \ref{def: Campana points and Darmon points}. In particular, we will clarify the relation of Conjecture \ref{conj: generalized Manin conjecture} with the conjecture \cite[Conjecture 1.1]{PSTVA21} on Campana points of bounded height.
	In order to uniformly discuss these different kinds of points, we introduce the notion of a quasi-Campana pair.
	\begin{definition}
		Let $m_1,\dots,m_n\in \mathbb{N}\cup \{\infty\}$.
		A pair $(X,M)$ is \textit{quasi-Campana} for the Campana pair $(X,D_{\mathbf{m}})$, where $D_{\mathbf{m}}=\sum_{i=1}^n \left(1-\frac{1}{m_i}\right)D_i$, if the following conditions are satisfied:
		\begin{enumerate}
			\item for all $i\in \{1,\dots,n\}$ with $m_i=\infty$ we have $w_i=0$ for all $(\mathbf{w},c)\in \fM_{\cC}$,
			\item $(m_i\mathbf{e}_i,D_i)\in \fM_{\cC}$ for all $i\in \{1,\dots,n\}$ with $m_i<\infty$, where $\mathbf{e}_i$ is the $i$-th standard basis vector of $\mathbb{Z}^n$,
			\item for all $(\mathbf{w},c)\in \fM_{\cC}\setminus \{(\mathbf{0},X)\}$, $\sum_{i=1}^n \frac{w_i}{m_i}\geq 1.$
		\end{enumerate}
	\end{definition}
	Examples of quasi-Campana pairs are given by the pairs for Campana points and Darmon points, their geometric counterpoints, as well as weak Campana points. In the theory of Campana points, the log-canonical class $K_X+D_{\bfm}$ plays a crucial role. We give an intrinsic definition of the log-canonical class on a pair $(X,M)$ as the ``best approximation from below'' of the canonical class $K_{(X,M)}$ by a $\mathbb{Q}$-divisor class on $X$.
	\begin{definition}
		Let $(X,M)$ be a smooth proper pair over a field $K$ such that $X$ is rationally connected. A $\mathbb{Q}$-divisor class $D$ on $X$ is called the \textit{log-canonical class} for $(X,M)$ if $K_{(X,M)}-\pr_M^*D\in \Eff^1(X,M)$ and for any $\mathbb{Q}$-divisor class $D'$ satisfying the same property, we have $D-D'\in \Eff^1(X)$. We will write $K_{(X,M),\log}$ for the log-canonical class.
	\end{definition}
	Note that since $\Eff^1(X,M)$ is strictly convex by Proposition \ref{prop: effective cone strictly convex}, any two divisors whose classes are log-canonical are $\mathbb{Q}$-linearly equivalent, so the above definition makes sense.
	
	For the pair corresponding to Campana points, the log-canonical divisor is simply $K_X+D_{\bfm}$.
	\begin{proposition}
		Let $(X,M)$ be a smooth proper quasi-Campana pair for the Campana pair $(X,D_{\mathbf{m}})$. Then $K_{(X,M),\log}=K_X+D_{\mathbf{m}}$ is the log-canonical divisor class for $(X,M)$. Furthermore, if $(X,M)$ is a pair corresponding to Darmon points, then $\pr_M^* K_{(X,M),\log}=K_{(X,M)}$.
	\end{proposition}
	\begin{proof}
		The definition of the log-canonical class $K_{(X,M)}$ immediately implies that $K_{(X,M)}-\pr_M^*(K_X+D_{\mathbf{m}})$ is an effective $\mathbb{Q}$-divisor class on $(X,M)$. Additionally, if $(X,M_{\text{Darmon}})$ is the pair corresponding to Darmon points on $(X,D_{\mathbf{m}})$ then $\pr^*_{M_{\text{Darmon}}} (K_X+D_{\mathbf{m}})= K_{(X,M_{\text{Darmon}})}$, so $K_X+D_{\mathbf{m}}$ is the log-canonical class for $(X,M_{\text{Darmon}})$. For other quasi-Campana pairs for the Campana pair $(X,D_{\mathbf{m}})$, the inclusion $\Gamma_{M_{\text{Darmon}},\cC}=\{(m_1\bfe_1,D_1),\dots,(m_n\bfe_n,D_n)\}\subset \Gamma_{M,\cC}$ induces a group homomorphism $\Pic(X,M)\rightarrow \Pic(X,M_{\text{Darmon}})$ compatible with the pullback homomorphisms $\pr^*_M$ and $\pr^*_{M_{\text{Darmon}}}$. This map sends effective divisor classes to effective divisor classes and $K_{(X,M)}$ to $K_{(X,M_{\text{Darmon}})}$. Thus for any $\mathbb{Q}$-divisor class $D$ on $X$ such that $K_{(X,M)}-\pr^*_M D$ is effective, $K_{(X,M_{\text{Darmon}})}-\pr^*_{M_{\text{Darmon}}} D$ is effective as well. This implies $D_{\bfm}-D\in \Eff^1(X)$, as we saw that $K_X+D_{\mathbf{m}}$ is the log-canonical class for $(X,M_{\text{Darmon}})$. Therefore $K_X+D_{\mathbf{m}}$ is the log-canonical class for $(X,M)$.
	\end{proof}
	There are various other pairs for which the log-canonical class exists, as the next example shows.
	\begin{example}
		Let $(X,M)$ be a smooth proper pair such that $\Pic(X)_{\mathbb{Q}}\cong \mathbb{Q}$ and such that $\Eff^1(X,M)$ is a rational polyhedral cone. Then there exists a log-canonical divisor on $X$ for $(X,M)$. This is because for any nonzero effective divisor $D$ on $X$ there is a largest $a\in \mathbb{Q}$ that satisfies $K_{(X,M)}-\pr_M^*a[D]\geq 0$, so $K_{(X,M),\log}=a[D]$ is the log-canonical class. Here the assumption on the effective cone ensures that $aD$ is a $\mathbb{Q}$-divisor, rather than just an $\mathbb{R}$-divisor.
	\end{example}
	However, the log-canonical class need not exist if the pair is not quasi-Campana, as the next example shows.
	\begin{example} \label{example: no canonical divisor}
		Let $X=\Bl_{(0:0:1)} \mathbb{P}^2$. Let $D_1=D$ be the strict transform of a line passing through $(0:0:1)\in \mathbb{P}^2$ and let $D_2=E$ be the exceptional divisor. Let $(X,M)$ be the divisorial pair for which the set of multiplicities $\fM\subset \mathbb{N}^2$ is the monoid generated by $(3,0)$, $(0,3)$ and $(1,1)$. Then the effective cone of $X$ is generated by $D$ and $E$, and a $\mathbb{Q}$-divisor $D'=(-2-a)D+(-1-b)E=K_X+(1-a)D+(1-b)E$ satisfies $K_{(X,M)}-\pr^*_M D'\geq 0$ if and only if $3a\geq 1$, $3b\geq 1$ and $a+b\geq 1$. There is no solution $(a,b)$ to this system of inequalities with $a$ and $b$ simultaneously minimal. Therefore there does not exist a log-canonical class on $X$ for $(X,M)$. 
	\end{example}
	
	For determining the Fujita invariant of a divisor class, we can use the log-canonical class rather than the canonical class of the pair.
	\begin{proposition} \label{prop: Fujita invariant quasi-Campana}
		Let $(X, M)$ be a smooth proper pair over a field $K$. Assume that there exists a log-canonical class $K_{(X,M),\log}$ for $(X,M)$. Let $L$ be a big and nef $\mathbb{Q}$-divisor class on $X$. Then
		$$a((X,M),L)=\mathrm{inf}\{t\in \mathbb{R}\mid tL+K_{(X,M),\log}\in \overline{\Eff}^1(X)\}.$$
	\end{proposition}
	\begin{proof}
		Since $L$ is big, and thus an effective $\mathbb{Q}$-divisor class, the infimum of all $t\in \mathbb{R}$ such that $tL+K_{(X,M),\log}$ is effective is the same as the infimum of all $t\in \mathbb{R}$ such that $tL+K_{(X,M),\log}$ is pseudo-effective.
		For any $\mathbb{Q}$-divisor class $D$ on $X$, we have $D+K_{(X,M),\log}\in \Eff^1(X)$ if and only if $\pr_M^* D+K_{(X,M)}\in \Eff^1(X,M)$, by the definition of the log-canonical class applied to $D'=-D$. By taking $D=tL$, we find the desired identity for $a((X,M),L)$.
	\end{proof}
	Similarly, we can use the log-canonical class to determine whether a divisor class is adjoint rigid.
	\begin{proposition} \label{prop: rigid quasi-Campana}
		Let $(X,M)$ be a proper quasi-Campana pair and let $L$ be a big and nef $\mathbb{Q}$-divisor on $X$. Then $a((X,M),L)\pr_M^*L+K_{(X,M)}$ is rigid if and only if $a((X,M),L)L+K_X+D_{\bfm}$ is rigid.
	\end{proposition}
	\begin{proof}
		Let $D'\in \Div(X)$ be a representative of the canonical divisor $K_X$, and set $D=\pr^*_M  D'+\sum_{(\bfm,c)\in \Gamma_{M,\cC}} \left(-1+\sum_{i=1}^n m_i\right) \tilde{D}_{\bfm,c}$ and $D_{\log}=D'+D_{\bfm}$. Then $D$, $D_{\log}$ represent the $\mathbb{Q}$-divisor classes $K_{(X,M)}$ and $K_X+D_{\bfm}$ on $(X,M)$ and $X$, respectively. A direct calculation shows that $D-\pr^*_M D_{\log}$ is an effective divisor which vanishes on $U$ and on the divisors corresponding to the elements $(m_i\mathbf{e}_i,D_i)\in \fM_{\cC}$. Thus if $E\in \Div(X)_{\mathbb{Q}}$ and $\tilde{D}$ is either a prime divisor on $U$ or $\tilde{D}= \tilde{D}_{m_i\bfe_i}$ for $i\in \{1,\dots, n\}$, then the coefficient of $\tilde{D}$ in $\pr_M^*(D_{\log}+E)$ is equal to the coefficient of $\tilde{D}$ in $D+\pr_M^* E$. Since the coefficient of $\tilde{D}_{m_i\bfe_i}$ in $\pr_M^*(D_{\log}+E)$ is $m_i$ times the coefficient of $D_i$ in $D_{\log}+E$ for all $i\in \{1,\dots,n\}$, this implies that $D_{\log}+E$ is effective if and only if $D+\pr_M^*E$ is effective. Therefore, $D_{\log}+E$ is rigid if and only if $D+\pr^*_M E$ is rigid.
	\end{proof}
	However, in general the log-canonical class need not be adjoint rigid when it exists.
	\begin{example}
		Let $X=\mathbb{P}^2_K$ and let $D_1$ and $D_2$ be two distinct lines in $\mathbb{P}_K^2$. As in Example \ref{example: no canonical divisor} we let $\fM\subset \mathbb{N}^2$ be the monoid generated by $(3,0)$, $(0,3)$ and $(1,1)$. The canonical divisor class $K_{(X,M)}$ of the divisorial pair $(X,M)$ is represented by the divisor $$\tilde{D}=-\pr^*_M (-2D_1-D_2+D_1+D_2)-\tilde{D}_{(3,0)}-\tilde{D}_{(0,3)}-\tilde{D}_{(1,1)}= -4\tilde{D}_{(3,0)}-\tilde{D}_{(0,3)}-2\tilde{D}_{(1,1)}.$$
		The Picard group of the pair, $\Pic(X,M)\cong \mathbb{Z}^2\times \mathbb{Z}/3\mathbb{Z}$, is generated by the divisors $\tilde{D}_{(3,0)},\tilde{D}_{(0,3)},\tilde{D}_{(1,1)}$, where $3\tilde{D}_{(3,0)}$ is linearly equivalent to $3\tilde{D}_{(0,3)}$.
		Both $E_1=-\tfrac{4}{3}D_1-\tfrac{2}{3}D_2$ and $E_2=-\tfrac{5}{3}D_1-\tfrac{1}{3}D_2$ represent the log-canonical class $K_{(X,M),\log}$. Since $\pr_M^*E_1+\widetilde{D}$ and $\pr_M^*E_2+\widetilde{D}$ are effective, $E_1$ and $E_2$ are not adjoint rigid for the pair $(X,M)$.
	\end{example}
	
	\begin{remark} \label{remark: polyhedron quasi-Campana}
		For a quasi-Campana pair, the polyhedron in Remark \ref{remark: reduction to small generators divisorial} is simply given by
		$$P=\left\{\bfx\in [0,\infty)^{\Gamma_{M,\cC}}\; \middle\vert \; \sum_{i=1}^n \frac{x_i}{m_i}\geq 1\right\}.$$
	\end{remark}
	
	This description of the polyhedron implies that for quasi-Campana pairs the $b$-invariant can be computed using the effective cone of $X$, rather than having to use the full effective cone of $(X,M)$. This is done by replacing the canonical class with the log-canonical class, and by adding a correction factor to the $b$-invariant.
	\begin{proposition} \label{prop: b-invariant quasi-Campana}
		Let $(X,M)$ be a smooth proper quasi-Campana pair, where $X$ is a rationally connected variety. Assume that at least one of the following conditions hold:
		\begin{enumerate}
			\item $\Eff^1(X)=\overline{\Eff}^1(X)$,
			\item $(X,M)$ is a pair corresponding to geometric Darmon points or geometric Campana points.
		\end{enumerate}
		Let $L$ be a big and nef $\mathbb{Q}$-divisor class on $X$, and let $G=\gal(E/K)$, where $E$ is the splitting field of $(X,M)$. Then $$b(K,(X,M),L)=b(K,(X,D_{\bfm}),L)+b'(K,(X,M),L),$$
		where $b(K,(X,D_{\bfm}),L)$ is the codimension of the minimal face $\mathcal{F}$ of $\overline{\Eff}^1(X)$ containing $A=a((X,M),L)L+K_X+D_{\bfm}$, and
		$$b'(K,(X,M),L)=\# I_{M,\cC},$$
		where $I_{M,\cC}$ is the set of $(\mathbf{w},c)\in \Gamma_{M,\cC}/G$ satisfying $\sum_{i=1}^n \frac{w_i}{m_i}=1$, $\mathbf{w}\neq m_1\bfe_1,\dots, \mathbf{w}\neq m_n\bfe_n$ and such that the coefficient of $\tilde{D}_{(\mathbf{w},c)}$ in $\pr^*_M D$ is zero for all effective divisors $D$ whose class lies in $\mathcal{F}$.
		In particular, if $L= -K_X-D_{\bfm}$, then
		$$b(K,(X,M),L)=\rank \Pic(X)+\#\bigg\{(\mathbf{w},c)\in \Gamma_{M,\cC}/G \bigg\vert \sum_{i=1}^n \frac{w_i}{m_i}=1, \mathbf{w}\neq m_1\bfe_1,\dots, \mathbf{w}\neq m_n\bfe_n\bigg\}.$$
	\end{proposition}
	
	\begin{remark} \label{remark: description I_C}
		Note that if $\mathcal{F}\subset \overline{\Eff}^1(X)$ is contained in the cone generated by the classes of the divisors $D_1,\dots, D_n$, then an element $(\mathbf{w},c)\in \Gamma_{M,\cC}$ lies in $I_{M,\cC}$ if and only if 
		\begin{itemize}
			\item $\sum_{i=1}^n \frac{w_i}{m_i}=1$,
			\item $\mathbf{w}\neq m_1\bfe_1,\dots, \mathbf{w}\neq m_n\bfe_n$, and
			\item $w_i=0$ for all $i=1,\dots, n$ with $D_i\in \mathcal{F}$.
		\end{itemize}
		In particular, this is satisfied when $X$ is a toric variety and the divisors $D_1,\dots,D_n$ include the torus-invariant divisors on $X$.
	\end{remark}
	
	\begin{remark}
		The assumption on the effective cone in the Proposition is satisfied for many varieties $X$, including Mori dream spaces such as Fano varieties \cite[Corollary 1.3.2]{BCHM10}.
	\end{remark}
	
	\begin{proof}
		Let $E$ be the splitting field of $(X,M)$.
		By Lemma \ref{lemma: reduction to small generators} and the description of the polyhedron given in Remark \ref{remark: polyhedron quasi-Campana}, we can without loss of generality assume that $\sum_{i=1}^n \frac{w_i}{m_i}=1$ for all $(\mathbf{w},c)\in\Gamma_{M,\cC}$. As a consequence of this assumption, we have $\pr^*_M (K_X+D_{\bfm})=K_{(X,M)}$, as $\pr^*_M D_{\bfm}-\sum_{(\mathbf{w},c)\in \Gamma_{M,\cC}} \tilde{D}_{\mathbf{w},c}=0$.
		
		We denote by $\tilde{\mathcal{F}}_E$ the minimal face of $\overline{\Eff}^1(X_E,M)$ containing $a((X,M),L)\pr^*_M L+K_{(X,M)}$ and we let $\tilde{\mathcal{F}}$ be its restriction to $\overline{\Eff}^1(X,M)$. Similarly, we write $\mathcal{F}_E$ for the minimal face of $\overline{\Eff}^1(X_E)$ containing $A$.
		The group homomorphism $\pr^*_M\colon \Pic(X)\rightarrow \Pic(X,M)$ induces a map $\overline{\Eff}^1(X)\rightarrow \overline{\Eff}^1(X,M)$, which we will also denote by $\pr^*_M$. As $(X,M)$ is proper, this map is injective by Proposition \ref{prop: rank pic for proper pair}. This fact combined with the equality $\pr^*_M (K_X+D_{\bfm})=K_{(X,M)}$ implies that an $D\in\Eff^1(X)$ lies in $\mathcal{F}$ if and only if $\pr^*_M D\in \Eff^1(X,M)$. 
		For any $(\mathbf{w},c)\in \Gamma_{M,\cC}$ such that $\mu((\mathbf{w},c),D)>0$ for some effective divisor $D\in \Div(X_E,M)$ whose class is contained in $\mathcal{F}_E$, the divisor class $[\tilde{D}_{\mathbf{w},c}]\in \Pic(X_E,M)_{\mathbb{Q}}$ lies in $\tilde{\mathcal{F}}_E$, as $\pr^*_M [D_i]\in \pr^*_M \mathcal{F}_E\subset \tilde{\mathcal{F}}_E$.
		
		
		
		We write $\langle \mathcal{F} \rangle$ and $\langle \tilde{\mathcal{F}}\rangle$ for the vector spaces generated by the corresponding cones, and consider the linear map $$f\colon \Pic(X)_{\mathbb{R}}/\langle \mathcal{F} \rangle \rightarrow \Pic(X,M)_{\mathbb{R}}/\langle \tilde{\mathcal{F}}\rangle,$$
		induced by $\pr^*_M$. Since the inverse image of $\tilde{\mathcal{F}}$ under $\pr^*_M$ is $\mathcal{F}$, $f$ is an injective map. Since $$\dim \Pic(X)_{\mathbb{R}}/\langle \mathcal{F} \rangle=b(K,(X,D_{\bfm}),L)$$ and
		$$\dim \Pic(X,M)_{\mathbb{R}}/\langle \tilde{\mathcal{F}}\rangle=b(K,(X,M),L),$$
		we have
		$$b(K,(X,M),L)=b(K,(X,D_{\bfm}),L)+\dim \coker(f).$$
		If the only solutions $(\mathbf{w},c)\in \Gamma_{M,\cC}$ to $\sum_{i=1}^n \frac{w_i}{m_i}=1$ are of the form $\mathbf{w}=m_i\bfe_i$, then $\pr^*_M$ gives an isomorphism $\Pic(X)_{\mathbb{R}}\rightarrow \Pic(X,M)_{\mathbb{R}}$, so $f$ is an isomorphism, giving $$b(K,(X,M),L)=b(K,(X,D_{\bfm}),L).$$ In particular, this proves the lemma if the pair corresponds to geometric Campana points or geometric Darmon points.
		
		Now we assume that $\Eff^1(X)=\overline{\Eff}^1(X)$, and we will show that the dimension of the cokernel of $f$ is $\# I_{M,\cC}$.
		
		For every $i\in \{1,\dots,n\}$, we have $\pr^*_M D_i=\sum_{(\mathbf{w},c)\in \Gamma_{M,\cC}} w_i \tilde{D}_{\mathbf{w},c}$, which implies $$m_i \tilde{D}_{m\bfe_i}=-\sum_{\substack{(\mathbf{w},c)\in \Gamma_{M,\cC}\\ \mathbf{w}\neq m\bfe_i}} w_i \tilde{D}_{\mathbf{w},c}$$
		in $\coker(f)$. Thus, the cokernel has $\left\{\left[\sum_{\sigma\in G}\tilde{D}_{\sigma(\mathbf{w},c)}\right]\mid (\mathbf{w},c)\in I_{M,\cC}\setminus\{m_1\bfe_1,\dots, m_n\bfe_n\}\right\}$ as a generating set as a vector space. We will now show that this set is a basis.
		
		First we will show that none of these generators lie in $\tilde{\mathcal{F}}_E$.
		Consider $(\mathbf{w},c)\in \Gamma_{M,\cC}$ satisfying $\mu((\mathbf{w},c),D')=0$ for all effective divisors $D'$ whose class is contained in $\mathcal{F}$. If $D, D_X\in \Div(X)_{\mathbb{Q}}$ are representatives of $L$ and $K_X$ such that $a((X,M),L)D+D_X+D_{\bfm}$ is effective, then the coefficient of $\tilde{D}_{(\mathbf{w},c)}$ in $\pr^*_M (a((X,M),L)D+D_X+D_{\bfm})\in \Div(X,M)_{\mathbb{R}}$ is $$\mu((\mathbf{w},c), a((X,M),L)D+D_X+D_{\bfm})=0.$$
		As $\pr^*_M (K_X+D_{\bfm})=K_{(X,M)}$, this implies that the coefficient of $\tilde{D}_{\mathbf{w},c}$ is zero in every effective representative of $a((X,M),L)\pr^*_M L+K_{(X,M)}$, and we see in particular that $[\tilde{D}_{\mathbf{w},c}]$ does not lie in $\tilde{\mathcal{F}}_E$.
		
		Suppose that $$D'=\sum_{(\mathbf{w},c)\in I_{M,\cC}\setminus\{m_1\bfe_1,\dots, m_n\bfe_n\}}a_{\mathbf{w},c}\tilde{D}_{\mathbf{w},c}$$
		is a $\mathbb{Q}$-divisor on $(X,M)$ which can be written as
		$\pr^*_M E'+E_1-E_2$ for a $\mathbb{Q}$-divisor $E'$ on $X$ and two effective $\mathbb{Q}$-divisors $E_1,E_2$ on $(X,M)$ such that $[E_1],[E_2]\in \tilde{\mathcal{F}}$. By modifying $E'$ if necessary, we can assume that the restriction of $E_1,E_2$ to $\Div(U)_{\mathbb{Q}}\times \oplus_{i=1}^n \mathbb{Q}(\tilde{D}_{m_i\bfe_i,D_i})$ is trivial. Since the restriction of $E$ to $\Div(U)_{\mathbb{Q}}\times \oplus_{i=1}^n \mathbb{Q}(\tilde{D}_{m_i\bfe_i,D_i})$ is trivial as well, we must have $E'=0$. For any $(\mathbf{w},c)\in I_{M,\cC}$ and any effective $\mathbb{Q}$-divisor $D''$ on $(X,M)$ whose class lies in $\tilde{\mathcal{F}}$, the coefficient of $\tilde{D}_{\mathbf{w},c}$ in $D'$ is zero, as we have shown earlier in the proof. This implies $E_1,E_2=0$ so $E=0$. Thus $I_{M,\cC}$ is a basis of $\coker(f)$, so we have shown
		$$b(K,(X,M),L)=b(K,(X,D_{\bfm}),L)+\# I_{M,\cC},$$
		as desired.
	\end{proof}

	For Darmon points, Campana points and weak Campana points on a Campana pair $(X,D_{\bfm})$, the $b$-invariant will generally be larger than $b(K,(X,D_{\bfm}),L)$. As usual, we write $D_{\bfm,E}=\sum_{i=1}^n \left(1-\frac{1}{m_i}\right)D_i$ for smooth prime divisors $D_1,\dots,D_n$. Furthermore, we write $D_{\bfm}=\sum_{i=1}^{n'} \left(1-\frac{1}{m'_i}\right)D'_i$ for prime divisors
	$D'_1,\dots, D'_{n'}$ on $X$ and positive integers $m'_1,\dots, m'_{n'}$. For $i\in \{1,\dots,n'\}$, let $D'_{i,E}=\sum_{\alpha\in \cA_i} D'_{i,\alpha}$ be the decomposition of $D'_i$ into geometrically irreducible components over the splitting field $E$.
	\begin{itemize}
		\item For weak Campana points, $I_{M,\cC}$ is the set of all pairs $(\bfw,c)\in \mathbb{N}^n_{\cC}/G$ such that $\sum_{i=1}^n \frac{w_i}{m_i}=1$, $\bfw\neq m_1\bfe_1,\dots, m_n\bfe_n$, and $w_i=0$ for every divisor $D_i$ appearing in the support of the log-adjoint divisor $A$.
		\item For Campana points and Darmon points, the set $I_{M,\cC}$ is in bijection with $$\bigsqcup_{\substack{i=1 \\ [D_i]\not\in \mathcal{F}}}^{n'} I^i_{M,\cC}/G,$$
		where $I^i_{M,\cC}$ is the set of all $(\bfw,c)$, where $\bfw\in \mathbb{N}^{\cA_i}$ such that $\sum_{\alpha\in \cA_i} \bfw_{\alpha}=m'_i$ and $\bfw\neq m'_i\bfe_{\alpha}$ for any $\alpha\in \cA_i$ and $c$ is a component of $\cap_{\substack{\alpha\in \cA_i \\ w_{\alpha}>0}} D'_{i,\alpha}$. In particular, $I_{M,\cC}\neq \emptyset$ unless every divisor $D_i'$ is smooth or satisfies $[D_i']\in \mathcal{F}$.
	\end{itemize}
Finally we predict that a quasi-Campana pair $(X,M)$ corresponding to a log Fano Campana pair $(X,D_{\bfm})$ is rationally connected.
	\begin{conjecture} \label{conjecture: Fano rationally connected}
		Let $(X,D_{\mathbf{m}})$ be a Campana pair for which the log-anticanonical divisor is ample. Then any proper quasi-Campana pair $(X,M)$ for $(X,D_{\mathbf{m}})$ is rationally connected.
	\end{conjecture}
	This conjecture is a related to a conjecture by Campana \cite[Conjecture 9.10]{Cam11} on Campana rational connectedness for Campana pairs (see also \cite[Conjecture 1.4]{CLT24}), which implies Conjecture \ref{conjecture: Fano rationally connected} for any pair $(X,M)$ corresponding to the Campana points on $(X,D_{\mathbf{m}})$.

\section{The fundamental group and the Brauer group of a pair} \label{section: fundamental group and Brauer group}
In this section we introduce the fundamental group and the Brauer group of smooth pairs over a field of characteristic $0$. We use this Brauer group in our formulation of the leading constant in Conjecture \ref{conj: generalized Manin conjecture}.

\subsection{The fundamental group}
We will introduce the fundamental group of a smooth pair using the valuations $v_{(\bfm,c)}$ introduced in Definition \ref{def: valuation (m,c)}. For a smooth variety $X$ over a field $K$, the étale fundamental group $\pi_1(X)$ admits a description using the valuations of the function field. Indeed, if $f\colon Y\rightarrow X$ is a finite morphism of smooth varieties, then purity of the ramification locus \cite[Tag 0BMB]{Stacks} implies that $f$ is étale if and only if it is unramified at every prime divisor $D$ of $X$.
\begin{definition}
	Let $L/K$ be a field extension and let $v$ be a discrete valuation on $K$ with uniformizer $\pi$. In the integral closure $A$ of $\cO_v$ in $L$, the ideal $(\pi)$ factors as $$(\pi)=\prod_{i=1}^k \mathfrak{p}_i^{e_i}$$ for maximal ideals $\mathfrak{p}_i$ of $A$ and positive integers $e_i$. We call the integers $e_1,\dots,e_k$ the \textit{ramification indices of $L/K$ with respect to the valuation $v$}. If $e_1=\dots=e_k=1$, then we say that $L/K$ is \textit{unramified} with respect to $v$.
\end{definition}

Using this definition, the fundamental group $\pi_1(X)$ is the inverse limit $\varprojlim_L\gal(L/K(X))$, where $L$ runs over all extensions of $K(X)$ which are unramified with respect to (the valuations corresponding to) the prime divisors $D$ on $X$.

Inspired by this description, we introduce the fundamental group of a smooth pair using Galois extensions with restricted ramification.
\begin{definition}
	For a smooth pair $(X,M)$ over a field $K$, we say that a separable field extension $L/K(X)$ is \textit{unramified} with respect to $(X,M)$ if
	\begin{enumerate}
		\item it is unramified with respect to all valuations on $K(X)$ corresponding to prime divisors on $U$, and
		\item for every $(\bfm,c)\in \Gamma_{M,\cC}$, the ramification indices with respect to the valuation $v_{(\bfm/\gcd(\bfm),c)}$ divide $\gcd(\bfm)$.
	\end{enumerate}
\end{definition}
We define the fundamental group using these unramified extensions.
\begin{definition} \label{def: fundamental group pair}
	Let $(X,M)$ be a smooth pair over a field $K$ of characteristic $0$. The \textit{(étale) fundamental group} of $(X,M)$ is
	$$\pi_1(X,M)=\varprojlim_L \gal(L/K(X)),$$
	where $L/K(X)$ runs over all Galois extensions which are unramified with respect to $(X,M)$.
\end{definition}
If $(X,M)$ is a smooth pair corresponding to Darmon points, then \cite[Lemma 8.7.]{Moe24} and purity of the branch locus \cite[Tag 0BMB]{Stacks} together imply that the fundamental group $\pi_1(X,M)$ is the étale fundamental group of the corresponding root stack.

\subsection{The Brauer group of a pair}
In \cite{MiNaSt22}, Mitankin, Nakahara and Streeter introduced the Brauer group $\Br(X,D_{\bfm})$ for Campana pairs $(X,D_{\bfm})$ in their study of Darmon points. This Brauer group was then used by Chow, Loughran, Takloo-Bighash and Tanimoto in their prediction \cite[Conjecture 8.3]{CLTT24} for the leading constant in the analogue of Manin's conjecture for Campana points. This group is defined as
$$\Br(X, D_{\bfm})=\{b\in \Br(X)\mid m_i\partial_{D_i} (b)=0 \text{ for all }i\in \{1,\dots,n\} \text{ with }m_i<\infty\},$$
where $\partial_{D_i}$ is the Witt residue as defined in \cite[Definition 1.4.11(ii)]{CoSk21}. Note that this is the Brauer group of the corresponding root stack \cite[Remark 3.15]{MiNaSt22}.

We now introduce the Brauer group of a smooth pair $(X,M)$, which will generalize the Brauer groups of Campana pairs.
\begin{definition} \label{def: Brauer group pair}
	Let $(X,M)$ be a smooth pair over a field of characteristic $0$.
	For every $(\bfm,c)\in \mathbb{N}^n_{\cC}$, the \textit{residue at $\tilde{D}_{(\bfm,c)}$} $$\partial_{\bfm,c}\colon \Br K(X)\rightarrow H^1(K_{(\bfm/\gcd(\bfm),c)},\mathbb{Q}/\mathbb{Z})$$
	is $\partial_{\bfm,c}=\gcd(\bfm)\cdot \partial_{\bfm/\gcd(\bfm),c}$, where $K_{(\bfm/\gcd(\bfm),c)}$ is the residue field of the valuation $v_{(\bfm/\gcd(\bfm),c)}$ and $\partial_{\bfm/\gcd(\bfm),c}$ is the residue corresponding to the valuation $v_{\bfm/\gcd(\bfm),c}$, as defined in \cite[Definition 1.4.11(ii)]{CoSk21}.
\end{definition}

\begin{definition}
	Let $(X,M)$ be a smooth pair over a field $K$ of characteristic $0$. The \textit{Brauer group} of the pair $(X,M)$ is
	
	$$\Br(X,M)=\{b\in \Br(U)\mid \partial_{\bfm,c} (b)=0 \,\text{ for all } (\bfm,c)\in \Gamma_{M, \cC}\}$$
	and its \textit{algebraic Brauer group} is
	$$\Br_1(X,M)=\Br(X,M)\cap \Br_1 U,$$
	where $\Br_1 U:=\ker\left(\Br U\rightarrow \Br U_{\overline{K}}\right)$ is the algebraic Brauer group of $U$.
\end{definition}
\begin{remark}
	Note that when $(X,M)$ is the smooth pair corresponding to geometric Darmon points on a pair $(X,D_{\bfm})$, then
	$$\Br(X,M)=\Br(X,D_{\bfm}).$$
\end{remark}

\subsection{The algebraic Brauer group for pairs over toric varieties}
In this section we describe the algebraic Brauer group for pairs over toric varieties in terms of automorphic characters on the open torus, which we will use in Section \ref{section: powerful values norm forms} and in \cite{Moe25toric}.

Let $(X,M)$ be a smooth pair such that $X$ is a toric variety over a number field and such that the divisors $D_1,\dots,D_n$ are the torus-invariant prime divisors, defined over the splitting field $E/K$ of the dense torus. Note that such a pair is always divisorial, as any nonempty intersection of torus-invariant prime divisors is a toric variety itself and thus irreducible.
By using the correspondence between Brauer group elements on a torus and automorphic characters as in \cite{Lou18}, we will describe the algebraic Brauer group $\Br_1(X,M)$ in terms of characters if $X$ is a (rational) toric variety and the divisors are the torus-invariant divisors.

\begin{definition}
	An \textit{automorphic character} on a torus $T$ over a field $K$ is a continuous homomorphism $T(\mathbf{A}_K)/T(K)\rightarrow S^1$. We denote the set of automorphic characters of finite order on $T$ by $(T(\mathbf{A}_K)/T(K))^{\sim}$.
\end{definition}

Let $G=\gal(E/K)$ and let $N$ be the cocharacter lattice of the split torus $U_E\cong \mathbb{G}_{m,E}^d$, viewed as a $G$-module. Then $N^\vee=\Hom(U_E,\mathbb{G}_m)$, which maps to $\Div(X_E)$ by sending a rational function on $X_E$ to its divisor. 
The homomorphism $$N^\vee\rightarrow \Div(X_E)\xrightarrow{\pr^*_M} \Div(X_E,M)$$ of $G$-modules induces a homomorphism $T_{(X,M)}\rightarrow U$ of tori, where $T_{(X,M)}$ is the torus corresponding to the $G$-module $\oplus_{\bfm\in \Gamma_M} \mathbb{Z}(\tilde{D}_{\bfm})$. The torus $T_{(X,M)}$ splits up as $T_{(X,M)}=\bigoplus_{\bfm\in\Gamma_M/G} T_{\bfm}$, where $\bfm$ corresponds to the $G$-submodule of $\Div(X,M)$ generated by $\tilde{D}_{\bfm}$. Note that the torus $T_\bfm$ is isomorphic to the Weil restriction of $\mathbb{G}_m$ along the field extension $K_\bfm/K$, where $K_\bfm$ is the smallest field containing $K$ over which $\tilde{D}_\bfm$ is defined.

In particular, given an automorphic character $\chi\colon U(\mathbf{A}_K)\rightarrow S^1$, we obtain an automorphic character $\chi_\bfm\colon T_{\bfm}(\mathbf{A}_K)\rightarrow S^1$ for every $\bfm\in \mathbb{N}^n_{\red}/G$ (here we recall that $\mathbb{N}^n_{\red}$ is the set of all $\bfm$ such that $\cC_{\bfm}\neq \emptyset$). The following lemma generalizes \cite[Lemma 3.25]{ShSt24}.
\begin{lemma} \label{lemma: Brauer group toric pair}
	Let $X$ be a smooth toric variety over a number field $K$ with dense torus $U$. Let $(X,M)$ be a smooth pair such that the divisors $D_1,\dots,D_n$ are the geometric components of the boundary $X\setminus U$. Then
	$$\Br_1(X,M)/\B(U)\cong \Br K\times \{\chi\in (U(\mathbf{A}_K)/U(K))^{\sim}\mid \chi_{\bfm}=1\, \forall \bfm\in \Gamma_M/G\}.$$
	Here $\B(U)=\ker(\Br_1(U)\rightarrow \prod_{v\in \Omega_K}\Br_1(U_{K_v}))$, which is trivial when $X$ is rational. This isomorphism identifies the Brauer--Manin pairing with the pairing between $U(\mathbf{A}_K)$ and the corresponding automorphic characters.
\end{lemma}
\begin{proof}
	For a torus $T$ over $K$ we write $\Br_e U=\{b\in \Br_1 U\mid 1_T^*b=0\}$, where $1_T\colon \spec K\rightarrow T$ is the identity element, so that we have $\Br_1 T\cong \Br K\times\Br_e T$. We similarly write $\Br_e(X,M)=\Br(X,M)\cap \Br_1 U$.
	
	For $\bfm\in \mathbb{N}^n_{\red}$ there is an isomorphism of tori $T_{\bfm}\rightarrow T_{\bfm/\gcd(\bfm)}$ given by mapping $\tilde{D}_{\bfm}$ to $\tilde{D}_{\bfm/\gcd(\bfm)}$, and under this identification we have $\chi_{\bfm}=\chi_{\bfm/\gcd(\bfm)}^{\gcd(\bfm)}$.
	For $\bfm\in \mathbb{N}^n_{\red}$ with $\gcd(\bfm)=1$ we can consider a smooth refinement of the fan of $X$ containing the ray $\mathbb{R}_{\geq 0}\phi(\bfm)$, giving a toric variety $X'$. If $D$ is the prime divisor on $X'$ corresponding to this ray, then the corresponding residue $\partial_{D}$ coincides with the residue $\partial_{\bfm}$. Similarly, the Galois module in $\Div(X'_{\overline{K}})$ generated by the geometric irreducible components of $D$ is naturally isomorphic to the $G$-module in $\Div(X_{\overline{K}},M)$ generated by the Galois conjugates of $\tilde{D}_{\bfm}$ and this isomorphism respects the pullback morphisms $\Div(X_{\overline{K}})\rightarrow \Div(X'_{\overline{K}})$ and $\Div(X_{\overline{K}})\rightarrow \Div(X_{\overline{K}},M)$.
	
	These observations combine with \cite[Lemma 4.7]{Lou18} to give a commutative diagram
	$$
	\begin{tikzcd}
		0 \arrow[r] & \B(U) \arrow[d] \arrow[r] & \Br_e(U) \arrow[r] \arrow[d]                                   & (U(\mathbf{A}_K)/U(K))^{\sim} \arrow[r] \arrow[d]                                & 0 \\
		& 0 \arrow[r]               & \bigoplus_{\bfm\in \Gamma_M/G} \Br_e(T_{\bfm}) \arrow[r, "\sim"] & \bigoplus_{\bfm\in \Gamma_M/G} (T_\bfm(\mathbf{A}_K)/T_\bfm(K))^{\sim} \arrow[r] & 0
	\end{tikzcd}$$
	with exact rows, as well as a diagram
	$$\begin{tikzcd}
		\Br_e(U) \arrow[d] \arrow[r]                 & \bigoplus_{\bfm\in \Gamma_M/G}\Br_e(T_\bfm) \arrow[d, "\sim"]                  \\
		\Br_1(U) \arrow[r, "\bigoplus\partial_\bfm"] & {\bigoplus_{\bfm\in \Gamma_M/G}H^1(K_{\bfm/\gcd(\bfm)},\mathbb{Q}/\mathbb{Z})},
	\end{tikzcd}$$
	which commutes up to sign. These diagrams together imply that for every $b\in\Br_e(U)$ and $\bfm\in \Gamma_M/G$, the residues $\partial_{\bfm}(b)$ vanishes if and only if the character $\chi_\bfm$ is trivial, where $\chi$ is the character on $U$ induced by $b$.
	
	Finally, the triviality of $\B(U)$ when $X$ is rational is proven in \cite[Corollary 4.6]{Lou18}.
\end{proof}
\section{The leading constant} \label{section: leading constant general}
In this section we give a description for the leading constant in Conjecture \ref{conj: generalized Manin conjecture} when the divisor class $L$ is adjoint rigid and the fundamental group $\pi(X_{\overline{K}},M)$ of the pair is abelian. Our expression for the leading constant will involve the Fujita invariant and the $b$-invariant, as well as analogues of the $\alpha$-constant and the Tamagawa constant for pairs. In Manin's conjecture, the complement of the support of the adjoint divisor is used to define several of these invariants.

The $\alpha$-constant and the Tamagawa constant in Manin's conjecture, as formulated in \cite{BaTs98}, are defined using the complement $X^\circ$ of the adjoint divisor. Our conjecture replaces the role of this open subvariety with a pair $(X^\circ,M^\circ)\subset (X,M)$.
\begin{definition}
	In the setup of Conjecture \ref{conj: generalized Manin conjecture} such that $L$ is adjoint rigid, let $A\in \Eff^1(X,M)$ be the unique effective representative of the adjoint divisor class $a((X,M),L)\pr^*_M L+K_{(X,M)}$. We define $X^\circ$ to be the open subvariety of $X$ given as the complement of the closed subset $\overline{\Supp(A)|_U}$ in $X$ and let $M^\circ$ be given such that $\Gamma_{M^\circ,\cC}$ is the set of all $(\bfm,c)\in \Gamma_{M,\cC}$ such that the coefficient of $D_{(\bfm,c)}$ in $A$ is $0$.
\end{definition}
Note that $a((X,M),L)=a((X^\circ,M^\circ),L)$ and $b(K,(X,M),L)=b(K,(X^\circ,M^\circ),L)=\rank \Pic(X^\circ,M^\circ)$ as $\pr^*_{M^\circ} L$ is a multiple of the canonical divisor class $K_{(X^\circ,M^\circ)}$.
\begin{definition} \label{def: alpha-constant}
	The \textit{$\alpha$-constant} of the pair $(X,M)$ with respect to $L$ is
	$$\alpha((X,M),L):=\frac{1}{\#\Pic(X,M^{\circ})_{\mathrm{torsion}}} \int_{\Lambda^\vee} e^{-\langle \pr_{M^\circ}^*(L),\bfx\rangle}\d \bfx,$$
	where $\Lambda^\vee\subset \Pic(X,M^\circ)^\vee_{\mathbb{R}}$ is the dual of the effective cone $\Lambda=\Eff^1(X,M^\circ)$ and the integral is taken with respect to the Lebesgue measure on $\Pic(X,M^\circ)^\vee_{\mathbb{R}}$, normalized by the lattice $\Pic(X,M^\circ)^\vee\subset \Pic(X,M^\circ)^\vee_{\mathbb{R}}$.
\end{definition}

We will define the Tamagawa constant using the \textit{space of adelic $S$-integral $\cM$-points}
$$(\cX,\cM)(\mathbf{A}_{\cO_S}):=\prod_{v\in \Omega_K\setminus S} (\cX,\cM)(\cO_v)\times \prod_{v\in S} (X,M)(K_v),$$
and the $\gal(\overline{K}/K)$-module $\Pic(X^\circ_{\overline{K}},M^\circ)$. Let $L(\Pic(X^\circ_{\overline{K}},M^\circ),s)$ be associated the Artin $L$-function with Euler product $L(\Pic(X^\circ_{\overline{K}},M^\circ),s)=\prod_{v\in \Omega_K} L_v(\Pic(X^\circ_{\overline{K}},M^\circ),s)$, where we set $L_v(\Pic(X^\circ_{\overline{K}},M^\circ),s)=1$ for every infinite place $v$. We write $$L^*(\Pic(X^\circ_{\overline{K}},M^\circ),1)=\lim_{s\rightarrow 1} (s-1)^{b(K,(X,M),L)} L(\Pic(X^\circ_{\overline{K}},M^\circ),s),$$
which is a positive real number, and we also write $\lambda_v=L_v(\Pic(X^\circ_{\overline{K}},M^\circ),1)$.
Fix an adelic metrisation $\cK_X$ of the canonical divisor class $K_X$.
We endow the space of adelic $S$-integral $\cM$-points with the measure
\begin{equation}
	\tau_{(X^\circ,M^\circ)}:= L^*(\Pic(X^\circ_{\overline{K}},M^\circ),1) \prod_{v\in \Omega_K} (\lambda_v^{-1} \tau_{X,v}),
\end{equation}
where $\tau_{X,v}$ is the measure on $X(K_v)$ as defined in \cite[\S 2.1.8]{CLTs10} induced by the chosen metrisation of $K_X$.
Note that the measure $\tau_{(X^\circ,M^\circ)}$ can be viewed as an analogue of the Tamagawa measure defined in \cite[Theorem 1.1]{CLTs10}.

As in \cite[Conjecture 8.3]{CLTT24}, we define the Tamagawa constant as a sum over Brauer classes in $\Br(X^\circ,M^\circ)/\Br K$.
\begin{definition} \label{def: Tamagawa constant}
	For each $b\in \Br(X^\circ,M^\circ)$, we set
	$$\hat{\tau}(b)=L^*(\Pic(X^\circ_{\overline{K}},M^\circ),1) \prod_{v\in \Omega_K} \lambda_v^{-1}\int_{x_v\in (\cX,\cM)(\cO_v)} \frac{e^{2\pi i \inv_v b(x_v)}\d \tau_{X,v}}{H_{v,a((X,M),L)\cL+\cK_X}(x_v)},$$
	where $\inv_v\colon \Br K_v\rightarrow \mathbb{Q}/\mathbb{Z}$ is the local invariant.
	The \textit{Tamagawa constant} is
	$$\tau(K,S,(\cX,\cM),\cL)=\sum_{b\in \Br_1(X^\circ,M^\circ)/\Br K} \hat{\tau}(b).$$
\end{definition}
Note that constant $\hat{\tau}(b)$ only depends on the class of $b$ in $\Br(X^\circ,M^\circ)/\Br K$, so the summands in the Tamagawa constant are well defined.

Similarly to the constant in \cite[Conjecture 8.3]{CLTT24}, it is not clear whether the sum defining the Tamagawa constant converges, and if it does, whether it converges absolutely. In the case that the sum does not converge absolutely, the sum $\sum_{b\in \Br(X^\circ,M^\circ)/\Br K}\hat{\tau}(b)$ should be interpreted as the limit $$\lim_{B\subset \Br(X^\circ,M^\circ)/\Br K}\sum_{b\in B}\hat{\tau}(b),$$ where the limit ranges over all finite subgroups $B$ of $\Br(X^\circ,M^\circ)/\Br K$.

\begin{remark} \label{remark: adelic integral U}
	Note that the adelic integral $\hat{\tau}(b)$ can alternatively be written as
	\begin{equation}
		\hat{\tau}(b)=\int_{x\in U(\mathbf{A}_K)\cap (\cX,\cM)(\mathbf{A}_{\cO_S})} \frac{e^{2\pi i \ev_b(x)}\d \tau_{(X^\circ,M^\circ)}}{H_{a((X,M),L)\cL+\cK_X}(x)},
	\end{equation}
	where $\ev_b(x)=\sum_{v\in \Omega_K} \inv_v(b(x))$ is the Brauer--Manin pairing of $x$ and $b$.
	For a finite subgroup $B\subset \Br(U)$, we write $U(\mathbf{A}_K)^B$ for the set of adelic points on $U$ which pair trivially with all elements in $B$.
	Using this pairing together with character orthogonality, we can give an alternative description of the Tamagawa constant:
	\begin{align*}
		\tau(K,S,(\cX,\cM),\cL)=& L^*(\Pic(X^\circ_{\overline{K}},M^\circ),1)\times\\ & \lim_{B\subset \Br_1(X^\circ,M^\circ)/\Br K} \#B \int_{x\in U(\mathbf{A}_K)^B \cap (\cX,\cM)(\mathbf{A}_{\cO_S})} \frac{\d \tau_{(X^\circ,M^\circ)}}{H_{a((X,M),L)\cL+\cK_X}(x)}.
	\end{align*}
\end{remark}
\begin{definition}
	The constant in Conjecture \ref{conj: generalized Manin conjecture} is defined to be
	\begin{equation} \label{equation: leading constant}
		c(K,S,(\cX,\cM),\cL)=\frac{\alpha((X,M),L)\tau(K,S,(\cX,\cM),\cL)}{a((X,M),L)(b(K,(X,M),L)-1)!}
	\end{equation}
\end{definition}

\begin{remark} \label{remark: alpha constant}
	In \cite{Pey95}, Peyre used a variant of the $\alpha$-constant, which is used in various articles on Manin's conjecture such as the work on toric varieties by Salberger, Pieropan and Schindler \cite{Sal98,PiSc23}. Rather than taking an exponential integral over the dual of the effective cone, this constant is instead defined as the volume of a slice of this cone. We define a generalization of Peyre's $\alpha$-constant, which we use in \cite{Moe25toric} to prove Conjecture \ref{conj: generalized Manin conjecture} for toric pairs. 
	
	Recall the cone $\Lambda=\Eff^1(X,M^\circ)$, and let $\Lambda^\vee_1\subset \Lambda^\vee$ be the collection of all linear functions in $\Pic(X^\circ,M^\circ)^\vee$ which evaluate to $1$ at the class $L$ and let $(\Lambda^\vee)^\circ$ be the interior of $\Lambda^\vee$. Then $ \mathbb{R}_{>0}\times \Lambda^\vee_1\rightarrow (\Lambda^\vee)^\circ$ given by $(c,f)\mapsto cf$ is an analytic isomorphism. We endow $\Lambda^\vee_1$ with the unique measure $\mu$ such that the measure on $\mathbb{R}_{>0}\times E$ corresponds to the measure on $\Lambda^\vee$ under this isomorphism, where we take the measure on $\mathbb{R}_{>0}$ to be the standard Lebesgue measure.
	If we write
	$$\alpha_{\mathrm{Peyre}}((X,M),L)= \frac{\Volume(\Lambda^\vee_1)}{\#\Pic(X,M^{\circ})_{\mathrm{torsion}}},$$
	then a general result on cones \cite[Proposition 2.4.4]{BaTs95} implies
	$$\alpha((X,M),L)=a((X,M),L)(b(\mathbb{Q},(X,M),L)-1)!\alpha_{\mathrm{Peyre}}((X,M),L).$$
\end{remark}

Using this alternate $\alpha$-constant, the constant in the conjecture is equal to
\begin{equation}
	c(K,S,(\cX,\cM),\cL)=\alpha_{\mathrm{Peyre}}((X,M),L)\tau(K,S,(\cX,\cM),\cL).
\end{equation}

\section{Compatibility with previous results} \label{section: Compatibility}
In this section we verify that Conjecture \ref{conj: generalized Manin conjecture} is compatible with various results in the literature.
\subsection{Compatibility with conjecture for Campana points}
First we show that Conjecture \ref{conj: generalized Manin conjecture} agrees with the conjecture \cite[Conjecture 1.1]{PSTVA21} by Pieropan, Smeets, Tanimoto and Várilly-Alvarado with the modified prediction for the leading constant given by Chow, Loughran, Takloo-Bighash and Tanimoto \cite[Conjecture 8.3]{CLTT24}.

Let $(X,D_{\bfm})$ be a Campana pair where $X$ is a smooth projective Fano variety over a number field such that the support of $D_{\bfm}$ is a strict normal crossings divisor, and let $(X,M)$ be the smooth split pair corresponding to the Campana points on this Campana pair. Let $S$ be a finite set of places and let $\cX$ be a regular integral model of $X$ over $\cO_S$. Let $(\cX,\cM)$ be the induced integral model of $(X,M)$ and let $(\cX,\cD_{\bfm})$ be the corresponding Campana pair.
By Proposition \ref{prop: Fujita invariant quasi-Campana} and Proposition \ref{prop: b-invariant quasi-Campana} it follows that $a((X,M),L)=a((X,D_{\bfm}),L)$ and $b(K,(X,M),L)=b(K,(X,D_{\bfm}),L)$ for any big and nef divisor class $L$, where the right hand sides are defined as in \cite[Conjecture 1.1]{PSTVA21}. Therefore Conjecture \ref{conj: generalized Manin conjecture} agrees with \cite[Conjecture 1.1]{PSTVA21}, up to a differing prediction for the leading constant.
Assume as in \cite{CLTT24} that the support of $A=a((X,M),L)L+K_X+D_{\bfm}$ is contained in the support of $D_{\bfm}$. Then $(X^\circ,M^\circ)=(X,M^\circ)$ is the pair corresponding to Darmon points on $(X\setminus \Supp(A), D_{\bfm}|_{X\setminus \Supp(A)})$, where $\Supp(A)$ is the support of $A$. In particular, it follows that $\Br_1(X,M^\circ)=\Br_1((X,D_{\bfm}),L)$, where the right hand side is defined as in \cite[Definition 8.1]{CLTT24}. Furthermore, a direct calculation of the $\alpha$-constants show that $\alpha((X,M),L)=\alpha((X,D_{\bfm}),L)$, where the right hand side is defined as in \cite[\S 3.3]{PSTVA21}. These equalities together combine imply that our prediction of the leading constant agrees with the one given in \cite[Conjecture 8.3]{CLTT24}.
\subsection{Darmon points on vector group compactifications}
In \cite[Theorem 1.2]{Ito25}, Ito proves Conjecture \ref{conj: generalized Manin conjecture} for Darmon points on vector group compactifications, where the divisors are chosen to be in the complement of the vector group, and $L$ is assumed to be adjoint rigid with respect to the pair. Although the expression for the leading constant given in his theorem does not precisely match with our description of the leading constant, a direct analogue of \cite[Lemma 9.3]{PSTVA21} shows that they are the same.

\subsection{Darmon points and generalized Fermat equations}
In \cite{Ara25}, Arango-Piñeros counted Darmon points of bounded height on projective space, motivated by the study of primitive solutions of generalized Fermat equations
$$Ax^a+By^b+Cz^c=0,$$
for positive integers $a,b,c$ satisfying
$$\tfrac{1}{a}+\tfrac{1}{b}+\tfrac{1}{c}>1.$$
The Campana pair that he studied is $(\mathbb{P}^1_{\mathbb{Z}},(1-\tfrac{1}{a})D_1+(1-\tfrac{1}{b})D_2+(1-\tfrac{1}{c})D_3),$
where $D_1, D_2, D_3$ are the points $(0:1),(1:1),(1:0)$ and the integers $a,b,c$ are as above. The $\mathbb{Z}$-Darmon points on this pair are the primitive tuples $(x:y)$ such that $|x|$ is an $a$-th power, $|x-y|$ is a $b$-th power and $|y|$ is a $c$-th power.
By \cite[Theorem 1.2]{Ara25}, the number of these points up to Weil height $B$ tends to
$$CB^{1-1/a-1/b-1/c},$$
for an explicit constant $C>0$.
This result can be viewed as a special case of Conjecture \ref{conj: generalized Manin conjecture}, as the log-canonical divisor is $$-K_{(\mathbb{P}^1_{\mathbb{Q}},D_{\bfm})}=-K_{\mathbb{P}^1_{\mathbb{Q}}}-\left(1-\tfrac{1}{a}\right)[D_1]-\left(1-\tfrac{1}{a}\right)[D_2]-\left(1-\tfrac{1}{c}\right)[D_3]=\left(1-\tfrac{1}{a}-\tfrac{1}{b}-\tfrac{1}{c}\right)[D_1],$$
so the Fujita invariant with respect to $[D_1]$ is $1-\tfrac{1}{a}-\tfrac{1}{b}-\tfrac{1}{c}$ while the $b$-invariant is $1$. Thus his result agrees with Conjecture \ref{conj: generalized Manin conjecture}, up to possibly having a different leading constant.

\subsection{Geometric Campana and geometric Darmon points on Toric varieties} \label{section: geometric Campana toric}
In \cite{ShSt24}, Shute and Streeter consider geometric Campana points and geometric Darmon points for Campana pairs $(X,D_{\bfm})$, where $X$ is a toric variety and $D_{\bfm}$ is torus-invariant. For such pairs, they prove an analogue of Manin's conjecture on the number of geometric Campana points and geometric Darmon points of bounded log-anticanonical height. The invariants appearing in their asymptotic are direct analogues of those appearing in the conjecture \cite[Conjecture 8.3]{CLTT24} for Campana points. In particular, their result \cite[Theorem 1.1]{ShSt24} agrees with the prediction given in Conjecture \ref{conj: generalized Manin conjecture}.

\subsection{Powerful values of norm forms} \label{section: powerful values norm forms}
Now we will show that Conjecture \ref{conj: generalized Manin conjecture} agrees with the results by Streeter \cite[Theorem 1.1]{Str21} on powerful values of norm forms. The other main result from that paper, \cite[Theorem 1.4]{Str21}, is also compatible with Conjecture \ref{conj: generalized Manin conjecture}, as shown in \cite[\S 8.5.2.]{CLTT24}.

Let $E/K$ be a Galois extension of number fields of degree $n$, with Galois group $G$, let $\omega$ be a $K$-basis of $E$ over $K$, and let $m$ be a positive integer. We write $N_\omega$ for the corresponding norm form and we let $Z(N_{\omega})\subset X:=\mathbb{P}^{n-1}_K$ be its zero locus.

Under the assumption that $n$ is prime or $\gcd(n,m)=1$, Streeter obtained an asymptotic for the number of weak Campana points over $\cO_S$ of bounded anticanonical height on the Campana pair $\left(X, \left(1-\frac{1}{m}\right) Z(N_\omega)\right)$ with respect to the model $\cX=\mathbb{P}^{n-1}_{\cO_S}$, for any finite set of places $S$ containing an explicit set $S(\omega)\subset \Omega_K$. Here the adelic metrization $\cK_X$ of $K_X$ used to define the height is the toric metric on the anticanonical divisor class as in \cite[Theorem 2.1.6]{BaTs95}, where we view $X$ as a toric variety using the anisotropic torus $X\setminus Z(N_{\omega})$.

Let $S\subset \Omega_K$ be a finite set of places containing the infinite places as well as the places which ramify in $E/K$. Furthermore, we let $(\mathbb{P}^{n-1}_K,M)$ be the smooth arithmetic pair for the weak Campana points on the Campana pair $\left(\mathbb{P}^n_K,\left(1-\frac{1}{m}\right) Z(N_\omega)\right)$, and take $(\mathbb{P}^N_{\cO_S},\cM)$ to be its natural integral model.

For any integer $m\in \mathbb{Z}$, we compute the asymptotic formula predicted by Conjecture \ref{conj: generalized Manin conjecture} and we show that it agrees with \cite[Theorem 1.1]{Str21} if $n$ is prime or $\gcd(n,m)=1$.
\begin{theorem} \label{Theorem: conjecture compatible with norm forms}
	Conjecture \ref{conj: generalized Manin conjecture} predicts that there exists a thin set $Z$ such that
	\begin{equation}
		N((\cX,\cM)(\cO_S)\setminus Z,\cK_X,B)\sim c B^{\frac{1}{m}} (\log B)^{b(K,(X,M),-K_X)-1},
	\end{equation}
	as $B\rightarrow \infty$, where
	$$b(K,(X,M),K_X)=\#(\{\bfw\in \mathbb{N}^n\mid \min(\bfw)=0, \, \sum_{i=1}^n w_i=m\}/G)$$
	and $G=\gal(E/K)$.
	The leading constant is given as
	$$c=\frac{L^*(\Pic(X_{\overline{K}},M^\circ),1)}{(b(K,(X,M),-K_X)-1)!n m^{b(K,(X,M),-K_X)-1} } \sum_{b\in \Br_1(X,M^\circ)/\Br K} \hat{\tau}(b),
	$$
	and
	$$\hat{\tau}(b)=\int_{x\in U(\mathbf{A}_K)\cap(\cX,\cM)(\mathbf{A}_{\cO_S})} e^{2\pi i \ev_b(x)}H_{-\cK_X}(x)^{1-\frac{1}{m}}\d \tau_{(X,M^\circ)}.$$ 
	
	Moreover, this prediction (with $Z=\emptyset$) agrees with \cite[Theorem 1.1]{Str21} if the assumptions of that theorem are satisfied.
\end{theorem}
\begin{proof}
	As before, we write $X= \mathbb{P}^{n-1}_K$.
	We start by computing the Fujita invariant and the $b$-invariant of the arithmetic pair $(\mathbb{P}^{n-1}_K,M)$. We first note that the divisor $Z(N_{\omega})$ has degree $n$ and hence represents the anticanonical divisor class $-K_X$, which implies that the log-anticanonical divisor class is represented by $-\frac{1}{m}K_X$, which in turn implies that the Fujita invariant is given by $a((X,M),-K_X)=\frac{1}{m}$. The pair $(X,M^\circ)$ satisfies $$\Gamma_{M^\circ}=\{\bfw\in \mathbb{N}^n\mid \min(\bfw)=0, \, \sum_{i=1}^n w_i=m\},$$ so $$b(K,(X,M),K_X)=\#(\{\bfw\in \mathbb{N}^n\mid \min(\bfw)=0, \, \sum_{i=1}^n w_i=m\}/G),$$ where the action of $G$ on $\mathbb{N}^n$ is induced by its action on $\{D_1,\dots,D_n\}$.
	
	Since the divisor $Z(N_{\omega})$ has degree $n$, the group $\left(\bigoplus_{\bfm\in \Gamma_{M^\circ}} \mathbb{Z}[D_\bfm]\right)^G\subset \Pic(X,M^\circ)$ is a subgroup of index $n$. As $\pr^*_{M^\circ} -K_X=-mK_{(X,M^\circ)}$ and $-K_{(X,M^\circ)}=\sum_{\bfw\in \Gamma_{M^\circ}}D_{\bfw}$, it follows that
	$$\alpha((X,M),-K_X)=\frac{1}{n m^{b(K,(X,M),-K_X)}}.$$
	
	Thus Conjecture \ref{conj: generalized Manin conjecture} implies that there exists a thin set $Z$ such that
	\begin{equation}
		N((\cX,\cM)(\cO_S)\setminus Z,\cK_X,B)\sim C(K,S,(\cX,\cM),\cK_X) B^{\frac{1}{m}} (\log B)^{b(K,(X,M),-K_X)-1},
	\end{equation}
	as $B\rightarrow \infty$, where
	$$
	c(K,S,(\cX,\cM),\cK_X)=\frac{L^*(\Pic(X_{\overline{K}},M^\circ),1)}{(b(K,(X,M),-K_X)-1)!n m^{b(K,(X,M),-K_X)-1} } \sum_{b\in \Br_1(X,M^\circ)/\Br K} \hat{\tau}(b),
	$$
	and
	$$\hat{\tau}(b)=\int_{x\in U(\mathbf{A}_K)\cap(\cX,\cM)(\mathbf{A}_{\cO_S})} e^{2\pi i \ev_b(x)}H_{-\cK_X}(x)^{1-\frac{1}{m}}\d \tau_{(X,M^\circ)}.$$
	Note that we used the description for  the $\hat{\tau}(b)$ from Remark \ref{remark: adelic integral U} here.
	
	
	Now we will relate this with the results obtained by Streeter. Let $\cK_X$ be the toric metrization of the canonical divisor as in \cite[Definition 4.5.]{Str21}. Under the additional assumption that $n$ and $m$ are coprime or $n$ is prime, \cite[Theorem 1.1]{Str21} states that 
	$$N((\cX,\cM)(\cO_S),\cK_X,B)\sim \widetilde{C} B^{\frac{1}{m}} (\log B)^{b(n,m)-1}$$
	for every finite set of places $S$ containing an explicit set $S(\omega)$,
	where $\widetilde{C}$ is an explicit positive constant and $$b(n,m)=\frac{1}{d}\left(\binom{n+m-1}{n-1}- \binom{m-1}{n-1}\right).$$ We will show that this agrees with our prediction.
	
	\textit{The b-invariant}:
	We first compare the exponent on the $\log B$ factor. In Streeter's work \cite[Proposition 5.8]{Str21}, he considers the set $G^m/S_m$, where $S_m$ acts by permuting the coordinates. The Galois group $G$ acts on $G^m/S_m$ by right multiplication of every element of a representative $m$-tuple. We choose an ordering on $G$.
	We have an isomorphism of $G$-sets $$G^m/S_m\cong \left\{\bfw\in \mathbb{N}^n\mid \sum_{i=1}^n w_i=m\right\}$$ given by sending $(g_1,\dots,g_m)$ to $(w_1,\dots, w_n)$, where $w_i$ is the number of times the $i$-th element of $G$ appears in the tuple.
	
	Under this isomorphism, the $G$-subset $\Gamma_{M^\circ}$ corresponds to the $G$-subset of $G^m/S_m$ given by all $m$-tuples not containing some element of $G$.
	As in \cite[Proposition 5.11]{Str21}, write $S'(G,m)\subset (G^m/S_m)/G$ for the quotient of this set by $G$. By the isomorphism, we see that $b(K,(X,M),-K_X)=\# S'(G,m)$. If we assume that $\gcd(m,n)=1$ or $n$ is prime as in \cite[Theorem 1.1]{Str21}, then \cite[Remark 5.12]{Str21} gives $b(K,(X,M),-K_X)=b(n,m)$. In particular, we have shown that the exponents on $B$ and $\log B$ in \cite[Theorem 1.1]{Str21} agree with our expectation.
	
	\textit{The leading constant}:
	It remains to verify the equality $\widetilde{c}=c(K,S,(\cX,\cM),\cK_X)$, where $\widetilde{c}$ is the constant given by Streeter.
	The equality $$\lim_{s\rightarrow \frac{1}{m}} \left(s-\frac{1}{m}\right)^{b(n,m)}L(\Pic(X_{\overline{K}},M^\circ),ms)=L^*(\Pic(X_{\overline{K}},M^\circ),1)\cdot \frac{1}{m^{b(n,m)}}$$
	implies that $c(K,S,(\cX,\cM),\cK_X)$ is the sum of the adelic integrals
	$$\frac{m}{(b(n,m)-1)!n}\lim_{s\rightarrow \frac{1}{m}} \left(s-\frac{1}{m}\right)^{b(n,m)}\int_{x\in U(\mathbf{A}_K)\cap(\cX,\cM)(\mathbf{A}_{\cO_S})} \frac{e^{2\pi i \ev_b(x)}\d \tau_{X}}{H_{-\cK_X}(x)^{s-1}}$$
	for all $b\in \Br_1(X,M^\circ)/\Br(K)$, where $\tau_X$ is the Tamagawa measure on $X$ determined by $\cK_X$. By \cite[Proposition 3.4.4.]{BaTs95}, this is equal to
	$$\frac{m\operatorname{Res}_{s=1}\zeta_K(s)}{(b(n,m)-1)!n\operatorname{Res}_{s=1}\zeta_E(s)}\lim_{s\rightarrow \frac{1}{m}} \left(s-\frac{1}{m}\right)^{b(n,m)}\int_{x\in U(\mathbf{A}_K)\cap(\cX,\cM)(\mathbf{A}_{\cO_S})} \frac{e^{2\pi i \ev_b(x)}\d \mu}{H_{-\cK_X}(x)^{s}}.$$
	Note that the factor $\frac{\operatorname{Res}_{s=1}\zeta_K(s)}{\operatorname{Res}_{s=1}\zeta_E(s)}$ appears here due to how the measure $\mu$ is normalized in \cite[\S 3.2]{Str21}. In order to prove that our constant $c$ agrees with the constant $\tilde{c}$, it only remains to relate our sum over Brauer elements with the sum over Hecke characters in \cite{Str21}.
	
	\textit{Hecke characters and the Brauer group:} 
	As $X$ is a rational toric variety with torus $U$, Lemma \ref{lemma: Brauer group toric pair} gives us an isomorphism $\Br_1(X,M^\circ)/\Br K$ with a subgroup $B$ of $(U(\mathbf{A}_K)/U(K))^\sim$.
	If we write $H$ for the set of automorphic characters on $U$, then
	$$\Br_1(X,M^\circ)/\Br K\cong B:=\{\chi\in H\mid \chi_\bfw=1\, \forall \bfw\in \Gamma_{M^\circ}/G\}.$$
	In particular, we see that
	$$c=\frac{m\operatorname{Res}_{s=1}\zeta_K(s)}{(b(n,m)-1)!n\operatorname{Res}_{s=1}\zeta_E(s)}\lim_{s\rightarrow \frac{1}{m}} \left(s-\frac{1}{m}\right)^{b(n,m)}\sum_{\chi\in B}\int_{x\in U(\mathbf{A}_K)\cap(\cX,\cM)(\mathbf{A}_{\cO_S})} \frac{\chi(x)\d \mu}{H_{-\cK_X}(x)^{s}}.$$
	This expression coincides with the expression given for $\tilde{c}$, except that the set of characters summed over in $\tilde{c}$ can a priori be different.
	
	The torus $U$ is given as $U=R_{E/K} \mathbb{G}_m/\mathbb{G}_m$, where $R_{E/K} \mathbb{G}_m$ is the Weil restriction. Thus, an automorphic character on $U$ can be viewed as an automorphic character on $R_{E/K} \mathbb{G}_m$ whose restriction to $\mathbb{G}_m$ is trivial. In other words, we can view them as Hecke characters for $E$ whose restriction to $\mathbf{A}^\times_K$ is trivial. The group $\gal(E/K)$ acts on the set of these Hecke characters by $\chi\mapsto \chi^\sigma$ for $\sigma\in \gal(E/K)$ as in \cite[Definition 2.23]{Str21}, where $\chi^\sigma(x)=\chi(\sigma(x))$ for all $x\in \mathbf{A}^\times_E$.
	We claim that $B=H[m]$ if $n=2$ and otherwise 
	$$B\subset B_0:=\{\chi\in H[m]\mid \chi^\sigma=\chi \quad\forall \sigma\in \gal(E/K)\},$$
	which we will now prove.
	
	For $\bfw\in \Gamma_{M^\circ}$, let $V_{\bfw}\subset \Div(X_{\overline{K}},M)$ be the Galois submodule generated by $\tilde{D}_{\bfw}$, so that its corresponding torus is $T_{\bfw}$. The natural homomorphism $T_{\bfw}\rightarrow U$ appearing in Lemma \ref{lemma: Brauer group toric pair} factors as $T_{\bfw}\rightarrow R_{E/K}\mathbb{G}_m\rightarrow U$, where the latter homomorphism is the projection and the former corresponds to the map $\mathbb{Z}[D_1,\dots,D_n]\rightarrow V_{\bfw}$. Let $V_{\bfw}\rightarrow \mathbb{Z}^{\gal(E/K)}$ be the Galois-equivariant homomorphism given by sending $\tilde{D}_{\bfw}$ to $$\sum_{\substack{\sigma\in \gal(E/K) \\ \sigma(\tilde{D}_\bfw)=\tilde{D}_\bfw}}\sigma.$$
	Using the isomorphism $\mathbb{Z}^{\gal(E/K)}\cong\mathbb{Z}[D_1,\dots,D_n]$ given by $\sigma\mapsto \sigma(D_1)$, we can view the composition $\mathbb{Z}[D_1,\dots,D_n]\rightarrow V_{\bfw}\rightarrow \mathbb{Z}^{\gal(E/K)}$ as an endomorphism $A_{\bfw}$ of $\mathbb{Z}[D_1,\dots,D_n]$. Since $A_{\bfw}$ sends $D_i$ to $\sum_{\sigma\in \gal(E/K)}w_{\sigma(i)}D_{\sigma(i)}$, we have $A_{\bfw_1+\bfw_2}=A_{\bfw_1}+A_{\bfw_2}$ for every $\bfw_1,\bfw_2\in \mathbb{N}^n$.
	From this we directly see that $\chi_{m\bfe_1}=1$ is equivalent to $\chi$ being $m$-torsion, as $V_{\bfe_1}\rightarrow \mathbb{Z}^{\gal(E/K)}$ is an isomorphism. Consequently we see that every element in $\Br_1(X,M^\circ)/\Br K$ is $m$-torsion, and furthermore $B=H[m]$ if $n=2$.
	
	Now we assume $n>2$. For every $\sigma\in \gal(E/K)$, the endomorphism $A_{\sigma(\bfe_1)}$ corresponds to the action of $\sigma$ on $R_{E/K}\mathbb{G}_m$, so the relations $\chi_{(m-1)\bfe_1+\sigma(\bfe_1)}=1$ and $\chi_{m\bfe_1}=1$ together imply that the characters in $\Br_1(X,M^\circ)/\Br K$ satisfy $\chi^\sigma=\chi$ if $n>2$. So we have the desired inclusion.
	
	Since the map $x\mapsto\prod_{\sigma\in \gal(E/K)}\sigma(x)$ is the norm map $\mathbf{A}_E\rightarrow \mathbf{A}_K$, it follows that $\prod_{\sigma\in \gal(E/K)}\chi^\sigma=1$ for all $\chi\in B_0$. As $\chi$ is Galois-invariant, it follows that $\chi$ is $n$-torsion. In particular, if $n$ and $m$ are coprime, then $B=B_0=1$ so $c=\tilde{c}$ as desired.
	
	If $n>2$ is prime, then $\gal(E/K)$ acts faithfully on the Galois modules $V_{\bfw}$ for $\bfw\in \Gamma_{M^\circ}$ so the map $V_{\bfw}\rightarrow \mathbb{Z}^{\gal(E/K)}$ is an isomorphism. This implies that a character $\chi$ on $U$ satisfies $\chi_{\bfw}=1$ if and only if the character $\chi^\bfw$ obtained by the map $R_{E/K} \mathbb{G}_m\rightarrow U$ satisfies $\chi^\bfw=1$. This implies $B=B_0$.
	If $n=2$ or $n>2$ then $B$ includes the set of characters $\cU_0$ in \cite[Definition 5.19]{Str21}, which are the characters in $B$ which are invariant under the compact subgroup $\cK$ as in \cite[Definition 4.17]{Str21}. Now \cite[Lemma 5.3]{Str21} implies that the only characters for which the integral does not vanish lie in $\cU_0$, which shows that $c=\tilde{c}$.

\end{proof}
\printbibliography
\end{document}